\documentclass[12pt]{amsart}
\usepackage{amsmath,amssymb,amsfonts,amsthm,amsopn}
\usepackage{graphicx,tikz}
\usepackage[all]{xy}
\usepackage{amsmath}
\usepackage{multirow, longtable, makecell, caption, array,enumitem}
\usepackage{amssymb}
\usepackage{mathtools}
\mathtoolsset{showonlyrefs}

\usepackage{bookmark}
\usepackage{hyperref,color}

\calclayout

\usepackage{geometry}
\newcommand{\adm}[1]{{\left\vert\kern-0.25ex\left\vert\kern-0.25ex\left\vert #1 
		\right\vert\kern-0.25ex\right\vert\kern-0.25ex\right\vert}}

\newtheorem{theorem}{Theorem}[section]
\newtheorem*{theorem*}{Theorem}
\newtheorem{lemma}[theorem]{Lemma}
\newtheorem{corollary}[theorem]{Corollary}
\newtheorem{proposition}[theorem]{Proposition}

\newtheorem{remark}[theorem]{Remark}

\theoremstyle{definition}

\usepackage{cellspace} %
\def\spann{\text{span}}
\def\bR{\mathbb{R}}
\def\bC{\mathbb{C}}
\def\bN{\mathbb{N}}
\def\bZ{\mathbb{Z}}

\def\bL{\mathbb{L}}
\def\bH{\mathbb{H}}
\def\smo{\setminus\{0\}}

\def\cF{\mathcal{F}}
\def\cA{\mathcal{A}}
\def\cV{\mathcal{V}}

\def\cS{\mathcal{S}}

\def\cL{\mathcal{L}}

\def\cI{\mathcal{I}}

\def\cP{\mathcal{P}}
\def\cQ{\mathcal{Q}}

\def\cH{\mathcal{H}}

\def\rd{\bR^d}

\def\rn{\bR^n}
\def\rnn{\bR^{2n}}

\def\rdd{\bR^{2d}}
\def\rdq{\bR^{4d}}

\def\lan{\langle}
\def\ran{\rangle}

\def\w{\mathrm{w}}

\def\S0{S^0_{0,0}}

\def\Bd'{B_{\delta'}}

\def\cBd'{\bar{B}_{\delta'}}

\def\a{\alpha}
\def\b{\beta}
\def\g{\gamma}
\def\d{\delta}
\def\vp{\varphi}

\def\irp{\int_0^{+\infty}}
\def\ird{\int_{\rd}}
\def\irdd{\int_{\rdd}}
\newcommand*{\M}[1]{M^{p_{#1},q_{#1}}}
\newcommand*{\W}[1]{W^{p_{#1},q_{#1}}}
\newcommand*{\X}[1]{X^{p_{#1},q_{#1}}}
\def\cFs{\cF_\sigma}
\def\Sp{\mathrm{Sp}(d,\bR)}
\def\Mp{\mathrm{Mp}(d,\bR)}
\newcommand*\dd[1]{\mathop{}\!\mathrm{d}#1}

\begin{document}
	
	\title[Phase space analysis of the twisted Laplacian]{Phase space analysis of spectral multipliers for the twisted Laplacian}
	
	\author[S. I. Trapasso]{S. Ivan Trapasso}
	\address{Dipartimento di Scienze Matematiche ``G. L. Lagrange'', Politecnico di Torino, corso Duca degli Abruzzi 24, 10129 Torino, Italy}
	\email{salvatore.trapasso@polito.it}
	
	
	\begin{abstract}
		We prove boundedness results on modulation and Wiener amalgam spaces for some families of spectral multipliers for the twisted Laplacian. We exploit the metaplectic equivalence relating the twisted Laplacian with a partial harmonic oscillator, leading to a general transference principle for the corresponding spectral multipliers. Our analysis encompasses powers of the twisted Laplacian and oscillating multipliers, with applications to the corresponding Schr\"odinger and wave flows. On the other hand, elaborating on the twisted convolution structure of the eigenprojections and its connection with the Weyl product of symbols, we obtain a complete picture of the boundedness of the heat flow for the twisted Laplacian. Results of the same kind are established for fractional heat flows via subordination. 
	\end{abstract}
	\maketitle
	
	\section{Introduction}
	
	In this note we deal with the phase space analysis of several flows stemming from the second-order partial differential operator on even-dimensional Euclidean spaces known as the \emph{twisted Laplacian}, or the \emph{special Hermite operator}. To be precise, setting $z=(x,y) \in \rdd$, we consider 
	\begin{equation}
		\cL \coloneqq - \sum_{j=1}^d \Big[\big(\partial_{x_j}-\frac{i}{2} y_j \big)^2+\big(\partial_{y_j}+\frac{i}{2}x_j\big)^2 \Big] = -\Delta_z + \frac{1}{4}|z|^2 -i\sum_{j=1}^d  \big(x_j\partial_{y_j} - \partial_{x_j}y_j\big).
	\end{equation}
	The study of this operator dates back at least to the works by Strichartz \cite{strichartz_89} and Thangavelu on Hermite and Laguerre expansions \cite{thangavelu_98,thangavelu_93,thangavelu_91}. The spectral analysis of $\cL$ is also a well-developed topic in the mathematical physics literature, where this operator is usually known as the \emph{Landau Hamiltonian}, as it governs the quantum dynamics of a charged particle under the influence of a uniform magnetic field. The twisted Laplacian has also received much attention over the years from the harmonic analysis community, due to its intimate connection with the sub-Laplacian on the Heisenberg group, that is $\bH^d = \bC^d \times \bR$ endowed with the product
	\begin{equation}
		(z,t) (w,s) = \Big( z+w,t+s+\frac12 \mathrm{Im}(z \cdot \overline{w}) \Big), 
	\end{equation} from which it can be derived via inverse Fourier transform with respect to the center. 
	
	Much is known on $\cL$ as a differential operator, especially concerning the local regularity aspects. It is a densely defined, essentially self-adjoint positive operator on $L^2(\rdd)$, which can be viewed as the pseudo-differential operator $L^\w$ with Weyl symbol
	\begin{equation} \label{eq-symbol-L}
		L(z,\zeta) = \sum_{j=1}^d \Big[\Big(\xi_j - \frac{y_j}{2} \Big)^2 + \Big(\eta_j+\frac{x_j}{2} \Big)^2\Big], \quad z=(x,y), \, \zeta=(\xi,\eta) \in \rdd. 
	\end{equation}
	The zero set of $L$ shows that $\cL$ is in fact a degenerate elliptic operator, meaning that it is not globally elliptic (in the sense of Helffer \cite{helffer}) in any H\"ormander symbol class. Nevertheless, $\cL$ is known to be \emph{globally regular} \cite{degoss_09,wong_05} --- that is, for a temperate distribution $f \in \cS'(\rdd)$, if $\cL f \in \cS(\rdd)$ then $f \in \cS(\rdd)$. This property happens to be satisfied by every Weyl operator $a^\w$ that is globally elliptic in the sense of Shubin \cite{shubin}, namely such that the symbol $a$ satisfies
	\begin{equation}
		|a(z,\zeta)| \ge C (1+|z|^2 + |\zeta|^2), \qquad |z|^2+|\zeta|^2 > R,
	\end{equation} for suitable constants $C,R>0$. Nevertheless, it is easy to realize that $\cL$ fails to be globally elliptic (or even hypoelliptic) in the context of the Shubin $\Gamma$-calculus or the Parenti-Cordes $G$-calculus.
	
	The careful analysis of twisted differential operators given in \cite{buzano} shows that, at least in dimension $d=2$, the global regularity of $\cL$ actually comes along with the global ellipticity and the injectivity in $\cS'(\rd)$ of its \textit{source}, that is the companion operator
	\begin{equation}
		\cH=- \sum_{j=1}^d \big[ \partial_{x_j}^2 + x_j^2 \big] = -\Delta + |x|^2. 
	\end{equation} This is the well known Hermite operator, or quantum harmonic oscillator --- the Weyl operator with symbol $H(q,p)= q^2+p^2$, $(q,p) \in \rdd$. Mentioning this result is barely scratching the surface of the deep entanglement between the twisted Laplacian and the harmonic oscillator. Indeed, if $\cH$ coincides with the quantization of $H$ according to the standard rule 
	\begin{equation}
		q_j \mapsto X_j=x_j, \quad p_j \mapsto D_j= -i\partial_{x_j},
	\end{equation} then $\cL$ can be similarly viewed as a non-standard quantization of $H$ according to the correspondence
	\begin{equation}
		q_j \mapsto V_j = -i\partial_{x_j}-\frac{1}{2}y_j, \qquad p_j \mapsto W_j = -i\partial_{y_j}+\frac{1}{2}x_j. 
	\end{equation}
	This form of quantization where differentiation and multiplication get intertwined has been widely explored in connection with quantum mechanics in phase space. A quite complete picture of the relation between standard Weyl and the so-called \emph{Landau-Weyl operators} has been given in \cite{degoss_10,degoss_09,degoss_08}, especially with regard to their spectral structure. Considering the twisted Laplacian and the harmonic oscillator for the sake of concreteness, it turns out that these operators have the same discrete spectrum 
	\begin{equation}
		\sigma(\cH)=\sigma(\cL)=d+2\bN \coloneqq \{d + 2k : k \in \bN\},
	\end{equation} but the eigenvalues of $\cL$ (the so-called Landau levels) are severely degenerate. To be precise, we have the resolutions of the identity
	\begin{equation}
		\cH = \sum_{k \in \bN} (d+2k) P_k, \qquad \cL = \sum_{k \in \bN} (d+2k) Q_k,
	\end{equation} where:
	
	\begin{itemize}\setlength\itemsep{0.25cm}
		\item[--] $P_k$ is the orthogonal projection onto the finite-dimensional eigenspace spanned by $d$-dimensional Hermite functions $\Phi_\a$, $\a \in \bN^d$, with $|\a|=\a_1+ \ldots +\a_d = k$.
		
		\item[--] The range of the eigenprojection $Q_k$ for the twisted Laplacian is spanned by the so-called special Hermite functions $\Phi_{\a,\b}$, $\a,\b \in \bN^d$, with $|\b|=k$, which are ultimately related to Laguerre polynomials \cite{thangavelu_93}.
	\end{itemize}
	
	A more detailed review of these properties can be found in Section \ref{sec-hermite-laguerre}. From a more abstract point of view, these spectral relationships come not as a surprise, since the operators involved in the Landau-Weyl quantization satisfy the canonical commutation relation $[V_j,W_k] = -i \delta_{j,k}$. As a result, the Stone-von Neumann theorem implies the unitary intertwining of the action of $\cL$ on $L^2(\rdd)$ with that of a partial harmonic oscillator $I \otimes \cH$ on $L^2(\rd)\otimes L^2(\rd)$, where $I$ is the identity operator. To be precise, there exists a unitary operator $\cA_J$ on $L^2(\rdd)$ such that $\cL = \cA_J (I\otimes \cH) \cA_J^*$. This quantitative connection is actually a fundamental ingredient in the characterization of the global regularity of $\cL$ given in \cite{buzano}.
	
	\subsection{Phase space analysis} The results discussed so far point out to two considerations at least. First, there are several bridges connecting the twisted Laplacian and the Hermite operator that could be possibly exploited --- for instance, establishing \textit{transference principles} (in both directions) is a traditional and powerful approach in this context \cite{thangavelu_93,vanN}. Secondly, the rise of $\cL$ as a twisted phase space quantization of the symbol $H$, as well as the investigations on the \textit{global} regularity of this operator, naturally call into play ideas and techniques of phase space analysis. 
	
	Aspects of both approaches are embraced in this note, where we perform a Gabor wave packet analysis of the twisted Laplacian and related flows --- such as those arising from the Schr\"odinger, heat or wave equation for $\cL$ and its fractional powers $\cL^{\nu}$, $0<\nu<1$. More precisely, we prove boundedness results for such semigroups on modulation spaces and Wiener amalgam spaces, which are families of Banach spaces characterized by the (mixed, possibly weighted) Lebesgue summability of a phase space representation of their members. In short, the Gabor transform of $f \in \cS'(\rd)$ with respect to the atom $g \in \cS(\rd)\smo$ is defined by 
	\begin{equation}
		V_g f (x,\xi) = (2\pi)^{-d/2} \ird e^{-i\xi \cdot y} f(y)\overline{g(y-x)} \dd{y}, \qquad (x,\xi) \in \rd \times \widehat{\rd} = \rdd,
	\end{equation} and it can be viewed as the (coefficient of a continuous) decomposition of $f$ into Gabor wave packets of the form $\pi(x,\xi) g (y) = (2\pi)^{-d/2} e^{i \xi \cdot y} g(y-x)$. 
	
	The modulation space $M^{p,q}(\rd)$, $1\le p,q \le \infty$, is the collection of the distributions satisfying
	\begin{equation}
		\| f \|_{M^{p,q}} = \Big( \ird \Big( \ird |V_g f(x,\xi)|^p \dd{x} \Big)^{q/p} \dd{\xi} \Big)^{1/q} < \infty,
	\end{equation} with obvious modifications in the case where $p=\infty$ or $q=\infty$. We write $M^p(\rd)$ if $q=p$. As a rule of thumb, the index $p$ essentially reflects the Lebesgue summability of $f$, while the index $q$ relates with the summability of its Fourier transform --- hence, with the regularity of $f$. Refinements can be provided by introducing suitable weights in the modulation space norms --- more details and generalizations as needed in this note can be found in Section \ref{sec-prel} below. 
	
	\subsection{A transference principle for spectral multipliers}
	We investigate the phase space structure of some spectral multipliers associated with $\cL$ and $\cH$. In general, given $m \in L^\infty(d+2\bN;\bC)$, we consider the \textit{spectral multipliers} defined as follows:
	\begin{equation}\label{eq:eigenexp-def}
		m(\cH) f \coloneqq \sum_{k \in \bN} m(d+2k) P_k f, \qquad m(\cL) g \coloneqq \sum_{k \in \bN} m(d+2k) Q_k g,
	\end{equation} where $f \in L^2(\rd)$ and $g \in L^2(\rdd)$. The study of their boundedness on Lebesgue or Hardy spaces is a classical problem of harmonic analysis, especially for the harmonic oscillator. Inspired by the analogous problem for Fourier multipliers, it is certainly interesting to investigate the behaviour of spectral multipliers for $\cH$ and $\cL$ on modulation spaces. Indeed, as far the standard Laplacian is concerned, it is remarkable that several dispersive flows associated with oscillating multipliers (including the Schr\"odinger and wave semigroups) preserve the time-frequency concentration of the initial datum --- in the sense that they are bounded on $M^p$, while they generally fail to be bounded on Lebesgue spaces except for $L^2=M^2$ \cite{benyi_unimod}. 
	
	Spectral multipliers of Mikhlin-H\"ormander type for the Hermite operator were proved to be continuous on suitable modulation spaces in \cite{bhimani_19}, where additional boundedness results for $m(\cH)$ were obtained via a reverse transference theorem involving the $L^p$ continuity of spectral multipliers for $\cL$ on the polarised Heisenberg group. Moreover, in the recent series of papers \cite{bhimani_21,bhimani_22,cordero_21} the authors focused on the time-frequency analysis of the fractional powers $\cH^\nu$ and the fractional heat propagator $e^{-t\cH^\nu}$, $\nu >0$. 
	
	The phase space analysis of $\cH$-multipliers performed in the aforementioned papers largely benefits from the Gabor analysis of pseudo-differential operators, that is nowadays widely developed --- especially for the Shubin $\Gamma$-classes, the symbol $H$ being globally elliptic in this sense. In order to obtain boundedness results for $m(\cL)$ we establish a transference principle on modulation spaces, by exploiting old and new results on Hermite multipliers in conjunction with the intertwining relationship  \begin{equation}\label{eq-intro-intertw}
		m(\cL) = \cA_J (I\otimes m(\cH))\cA_J^*,
	\end{equation} 
	which extends the unitary equivalence between the twisted Laplacian and a partial harmonic oscillator to the corresponding spectral functions. A simplified, unweighted form of the transference Theorem \ref{thm-transf} reads as follows.
	
	\begin{theorem*}
		If $m(\cH)$ is bounded on $M^p(\rd)$ for every $1 \le p \le \infty$, then $m(\cL)$ is bounded on $M^p(\rdd)$ for every $1 \le p \le \infty$.
	\end{theorem*}
	
	\noindent In particular, we are able to investigate the phase space behaviour of singular oscillating multipliers of the form
	\begin{equation}
		m_t(\cL) = \cL^{-\d/2}e^{it\cL^{\g/2}}, \quad \d \ge 0, \, \g \le 1, \, t>0,
	\end{equation} as detailed in Corollary \ref{cor-oscmult-L}. For $\d \in\{0,1\}$ and $\g=1$ we obtain the Poisson-type semigroups for $\cL$, for which only few results in the Euclidean setting are available to the best of our knowledge --- see for instance \cite{dancona} for refined dispersive estimates, \cite{tie} about the finite speed of propagation and \cite{thangavelu_93} for $L^p$ boundedness results. 
	
	The analysis of the oscillating multipliers for $\cL$ allows us to characterize the phase space regularity of the solutions of the wave equation $\partial^2_t u + \cL u = 0$ with initial data $u(z,0) = f(z)$ and $\partial_t u(z,0) = g(z)$, $(z,t) \in \rdd \times \bR$. These are given by
	\begin{equation}
		u(z,t) = \cos(t\cL^{1/2})f(z,t) + \cL^{-1/2}\sin(t\cL^{1/2})g(z,t), 
	\end{equation} and the time-frequency regularity of the initial data (as captured by the $M^p$ norm) is preserved by the wave flow, locally uniformly in time. More precisely, if $f,g \in M^p(\rdd)$, $1 \le p < \infty$ and $t \in [0,T]$ for some $T >0$, there exists $C(T)>0$ such that
	\begin{equation}
		\| u \|_{M^p} \le C(T) (\|f\|_{M^p}+\|g\|_{M^p}). 
	\end{equation}
	This result is obtained via a refined pseudo-differential analysis of the corresponding problem for the Hermite operator (Theorem \ref{thm-hermite-oscmult}), which is of independent interest due to the improvement of the currently available results given in \cite{bhimani_19}. In Section \ref{sec-fracpow} we are also able to use transference beyond the bounded functional calculus, proving boundedness results on (possibly weighted) modulation spaces for fractional powers $\cL^\nu$ with $\nu \in \bR$. 
	
	A key ingredient of the phase space transference principle behind these results is the full characterization of the unitary equivalence in \eqref{eq-intro-intertw}, which happens to be of metaplectic type. Indeed, as already realized in \cite{gramchev_10,gramchev_09}, $\cA_J$ is a metaplectic operator, ultimately related to the Fourier-Wigner transform linking Hermite and special Hermite functions. The action of metaplectic operators on modulation spaces is well understood, especially in light of their role with the Gabor analysis of the Schr\"odinger equation --- see e.g., \cite{CR_book,degoss_11_book} for a comprehensive overview. Indeed, metaplectic operators can be characterized as the Schr\"odinger flows $e^{-itF^\w}$ associated with the Weyl quantization $F^\w$ of quadratic classic Hamiltonian functions $F$. As such, boundedness results on modulation spaces for the (periodic) semigroups $e^{-it\cH}$ and $e^{-it\cL}$ fall within this framework --- see Proposition \ref{prop-metap-mod} and related comments in this connection.
	
	\subsection{The twisted convolution structure and the heat flow} While the transference approach is quite powerful, it should be highlighted that the tensor product structure of $I\otimes m(\cH)$ and the action of the metaplectic intertwiners combine into an unavoidable restriction of the scope of the results established for the Hermite multipliers --- including the obvious loss of smoothing effects or the range shrinkage of modulation spaces indices. These aspects are pretty manifest when transference is applied to negative powers of $\cH$ or the heat diffusion flows $e^{-t\cH^\nu}$, studied in full generality in \cite{bhimani_21}. 
	
	An undoubtedly better approach would rely on a refined time-frequency analysis of the twisted Laplacian and its spectral projections, which have a quite peculiar form: $Q_k f$ operates a \textit{twisted convolution} of $f$ with a suitable Laguerre function $\vp_k$ --- see Section \ref{sec-weyl-twist} for details, and the articles \cite{jeong,koch} for the highly non-trivial problem of obtaining sharp $L^p$ bounds for $Q_k$. It is reasonable to expect that such a twisted structure reverberates somehow into the spectral multipliers, but an explicit characterization of this phenomenon is contrasted by the occurrence of special functions and the limited knowledge on their time-frequency features. Nevertheless, in the case of the heat flow $e^{-t\cL}$, it is possible to obtain the integral representation 
	\begin{equation}
		e^{-t\cL}f(z) = (16\pi^2 \sinh t)^{-d} \irdd e^{-2i \sigma(z,w)} e^{-\frac14 \coth(t) |z-w|^2} f(w) \dd{w},  \quad t \ne 0,
	\end{equation} where $\sigma$ is the standard symplectic form on $\rdd$ --- that is, $\sigma(z,w) = z_2 \cdot w_1 - z_1 \cdot w_2$, for $z=(z_1,z_2)$ and $w=(w_1,w_2)$. This formula can be conveniently viewed as a twisted convolution with a suitable kernel, namely
	\begin{equation}
		e^{-t\cL} f = f \times p_t,  \qquad p_t(w) = (16\pi \sinh t)^{-d} e^{-\frac14 \coth(t) |w|^2}, 
	\end{equation}
	revealing the behind-the-scenes role of the twisted structure of the eigenprojections.
	
	The twisted representation of the heat flow for the special Hermite operator is particularly favourable in the context of phase space analysis. Indeed, the twisted convolution of functions relates via symplectic Fourier transform with the Weyl product of symbols (i.e., the bilinear form $\#$ such that $(a\#b)^\w = a^\w b^\w$) in pretty much the same way the Fourier transform turns the standard convolution into a pointwise product. The Gabor analysis of pseudo-differential operators, hence of their composition, has already led to a comprehensive characterization of the continuity of the Weyl product on modulation and amalgam spaces \cite{CTW,holst}, which can be conveniently exploited here to derive a full range of boundedness results for the twisted Laplacian heat flow $e^{-t\cL}$ (Theorem \ref{maint-heat1}). Boundedness results can be proved as well for the fractional heat flow $e^{-t\cL^\nu}$, $0 < \nu < 1$, essentially by means of subordination to the solutions of the standard heat flow (Theorem \ref{maint-heatfrac}). In short, our results read as follows.
	
	\begin{theorem*}
		For $0<\nu\le 1$ and $t>0$, the special Hermite heat semigroup $e^{-t\cL^\nu}$ is a continuous map $M^{p_1,q_1}(\rdd) \to M^{p_2,q_2}(\rdd)$, $1 \le p_1,q_1,p_2,q_2 \le \infty$, if and only if $q_2 \ge q_1$. In such case, the following bound holds:
		\begin{equation}
			\| e^{-t\cL^\nu}f \|_{M^{p_2,q_2}} \le C_\nu(t) \| f \|_{M^{p_1,q_1}}, 
		\end{equation} where
		\begin{equation}
			C_\nu(t) = \begin{cases}
				C e^{-td^\nu} & (t\ge 1) \\ C t^{-\lambda_\nu} & (0<t\le 1), 
			\end{cases} \qquad \lambda_\nu \coloneqq \max\Big\{ \frac{d}{\nu}\Big(\frac{1}{p_2}-\frac{1}{p_1}\Big),0 \Big\}, 
		\end{equation} for a constant $C>0$ that does not depend on $f$ or $t$. 
	\end{theorem*}
	
	Note that the decay for large times is sharp --- it is enough to test the bound in the ground state case where $f=\Phi_{0,0}$. Moreover, having in mind that the phase space effect of $\cL$ is roughly comparable to that of a partial harmonic oscillator, a singularity for small times is expected as well, as heuristically illustrated in \cite{bhimani_21} concerning the heat flow of $\cH$. Let us also stress that this heuristic link relying on metaplectic equivalence barely agrees with the occurrence of the condition $q_2 \ge q_1$ --- so that $e^{-t\cL^\nu}f$ can have at best the same regularity of $f$. The necessity of this constraint is actually due to the underlying twisted structure $e^{-t\cL^\nu} f = f\times p_t^{(\nu)}$. Even if an explicit formula for $p_t^{(\nu)}$ seems to be unachievable unless $\nu=1$, cheap time-frequency estimates for the involved Laguerre functions are enough to show that the fractional twisted heat propagator belongs to $M^1(\rdd)$. As a result, sharpness of the index constraint follows from the failure of the map $f \mapsto f \times p_t^{(\nu)}$ to be bounded $\M1(\rdd) \to \M2(\rdd)$ if $q_2 < q_1$ --- see Proposition \ref{prop-sharpiff} and \cite[Proposition 3.3]{holst} for details. These remarks illustrate that, as already anticipated, a phase space analysis relying on the twisted structure leads to a much more complete picture than the one that can be obtained via transference from the heat flows $e^{-t\cH^\nu}$.  
	
	The phase space continuity estimates for $e^{-t\cL^\nu}$ just illustrated extend via subordination and embeddings to Lebesgue spaces $L^p \to L^q$ with $1 \le p \le q \le \infty$, cf.\ Corollary \ref{cor-fracheatL-leb}. In the case where $p \le 2 \le q$ the bound reads
	\begin{equation}
		\| e^{-t\cL^\nu}f \|_{L^q} \le C_{p,q}^{(\nu)}(t) \| f \|_{L^p}, \qquad C_{p,q}^{(\nu)}(t) = \begin{cases}
			C e^{-td^\nu} & (t\ge 1) \\ C t^{-\frac{d}{\nu}(\frac{1}{p}-\frac{1}{q})} & (0<t\le 1), 
		\end{cases} 
	\end{equation} encompassing known contractivity results already proved in the case $\nu=1$ in \cite{wong_05,wong_03}. A subordination argument yields the same estimates for negative powers $\cL^{-\nu}$, $\nu > 0$. 
	
	\subsection{Comments and future work} A phase space analysis of spectral multipliers for the twisted Laplacian is the first step towards a thorough analysis of the corresponding evolution equations, possibly in presence of nonlinearities or potential perturbations. Inspired by the similar quest for the standard Schr\"odinger equation \cite{CR_book}, we expect to build upon the content of the present note to obtain a wide array of results in connection with problems of local/global well-posedness (also including Strichartz estimates) on modulation and amalgam spaces --- see \cite{bhimani_22,bhimani_21} for examples in this vein for the Hermite operator. 
	
	At a more fundamental level, a deeper understanding of the time-frequency distribution of the eigenfunctions of $\cL$ would be as desirable as it is challenging, but it could pave the way to many interesting results --- including sharp bounds on modulation spaces for the eigenprojections $Q_k$ in the spirit of \cite{jeong,koch} or extension of the boundedness results for oscillating multipliers to a broader range of modulation spaces, thus bypassing the restrictions coming along with transference. A related, far-reaching goal would be a full characterization of the space of $M^p$-bounded spectral multipliers for $\cL$. In fact, as already discussed, this problem is largely open even for the harmonic oscillator --- to the best of our knowledge, the results in \cite{bhimani_19} and the novel ones in Section \ref{sec-transf} are the only ones currently available, and there is reason to believe that they can be improved, at least in connection with the regularity of the multipliers (see Remark \ref{rem-mult-reg}) or the time dependence of the constants. We plan to pursue this intriguing quest in future work. 
	
	Another interesting problem would be to elucidate the connection between modulation spaces and regularity spaces of Lipschitz-H\"older type or atomic decomposition spaces of Besov and Triebel-Lizorkin type associated with the twisted Laplacian --- see e.g., \cite{cardona} for a recent example of their use to obtain sharp $L^p$ bounds for the Schr\"odinger equation. It is well known that modulation and amalgam spaces provide a robust family of functional spaces with non-trivial embeddings with standard Lebesgue, Sobolev or potential spaces \cite{CR_book} --- notably including those associated with the Hermite operator \cite{janssen}. 
	
	Finally, we believe that the phase space analysis of ``exotic'' differential operators and related flows can lead to interesting results from an original perspective, especially in the case where the standard machinery of pseudo-differential calculus on modulation spaces cannot be directly called into play. These obstacles stimulate the development of novel lines of attack. Besides the twisted Laplacian, there are several operators of interest in harmonic analysis (such as the Grushin operators \cite{grushin}) of which very little is known from the time-frequency side, which we plan to investigate in future work. 
	
	\section{Preparation}\label{sec-prel}
	
	\subsection{Notation} The set of natural numbers is $\bN=\{0,1,2\ldots\}$. We write $2\bN = \{ 2n : n \in \bN \}$ to denote the set of even numbers. We write $A \hookrightarrow B$ to emphasize that the embedding of $A$ into $B$ is continuous. 
	
	We make use of the symbol $X \lesssim_\lambda Y$ if the underlying inequality holds up to a positive constant factor (i.e., $X\le C_\lambda Y$) that does not depend on $X$ or $Y$ but may depend on the parameter $\lambda$. We write $X \asymp_\lambda Y$ if $X$ and $Y$ are of comparable size, namely both $X \lesssim_\lambda Y$ and $X\lesssim_\lambda Y$ hold. 
	
	We introduce a number of operators acting on $f\colon \rd \to \bC$: 
	\begin{itemize}
		\item[--] The translation $T_x$ by $x \in \rd$: $T_x f(y) = f(y-x)$.
		\item[--] The modulation $M_\xi$ by $\xi \in \rd$: $M_\xi f(y) = (2\pi)^{-d/2} e^{i \xi \cdot y} f(y)$.
		\item[--] The time-frequency shift $\pi(x,\xi) = M_\xi T_x$.
		\item[--] The reflection: $\cI f(y) = f^\vee(y) = f(-y)$.
	\end{itemize}
	The $L^2$ inner product $\lan f,g \ran = \ird f(y)\overline{g(y)}\dd{y}$ extends to a duality pairing between temperate distributions $f\in \cS'(\rd)$ and functions in the Schwartz class $g \in \cS(\rd)$, denoted by the same bracket. 
	
	Recall that the tensor product $f_1 \otimes f_2$ of two temperate distributions $f_1,f_2 \in \cS'(\rd)$ is the unique element of $\cS'(\rdd)$ such that 
	\begin{equation}
		\lan f_1\otimes f_2, g_1 \otimes g_2 \ran = \lan f_1,g_1 \ran \lan f_2,g_2 \ran, \quad \forall g_1,g_2 \in \cS(\rd),
	\end{equation} where $g_1\otimes g_2 \in \cS(\rdd)$ is defined as usual by $g_1 \otimes g_2(x,y) = g_1(x)g_2(y)$, $(x,y) \in \rdd$.  We remark here for future reference that, for all $z=(z_1,z_2)$ and $w=(w_1,w_2)$ in $\rdd$, 
	\begin{equation}\label{eq-tensor-pi}
		\pi(z,w)(g_1 \otimes g_2) = \pi(z_1,w_1)g_1 \otimes \pi(z_2,w_2)g_2.
	\end{equation}
	
	Given two continuous linear operators $T_1,T_2 \colon \cS'(\rd) \to \cS'(\rd)$, we denote by $T_1 \otimes T_2$ the unique operator on $\cS'(\rdd)\to \cS'(\rdd)$ such that
	\begin{equation}
		(T_1\otimes T_2)(u_1 \otimes u_2) = T_1u_1 \otimes T_2u_2, \quad \forall u_1,u_2 \in \cS'(\rd). 
	\end{equation}
	
	Recall that $\rdd$ is a symplectic vector space endowed with the standard form 
	\begin{equation}\label{eq-sympf}
		\sigma(z,w) \coloneqq Jz \cdot w, \qquad J\coloneqq \begin{bmatrix}
			O & I \\ -I & O
		\end{bmatrix},
	\end{equation} where $O$ and $I$ denote respectively the $d\times d$ null and identity matrix. In particular, if $z=(x,\xi)$ and $w=(u,v)$ belong to $\rd \times \rd$, then $\sigma(z,w)=\xi \cdot u - x \cdot v$. 
	
	\subsection{Tools from Gabor analysis}
	The Fourier transform is normalized as follows:
	\begin{equation}
		\cF f(\xi)= \hat{f}(\xi) \coloneqq (2\pi)^{-d/2} \ird e^{-i\xi \cdot x}f(x) \dd{x}.
	\end{equation}
	It is useful to introduce a (unitarily dilated) symplectic version of the Fourier transform for functions defined on even-dimensional spaces such as $\rdd$:
	\begin{equation}\label{eq-def-sympfou}
		\cFs f(\zeta) = f_\sigma(\zeta) \coloneqq 2^d\cF(2J\zeta) = \pi^{-d} \irdd e^{-2i\sigma(\zeta,z)}f(z) \dd{z}.
	\end{equation} Note that $\cFs^{-1}=\cFs$, hence $\cFs^2$ coincides with the identity operator. 
	
	The Gabor transform of $f \in \cS'(\rd)$ with respect to a window $g \in \cS(\rd)\smo$ is defined by
	\begin{equation}
		V_g f(x,\xi) \coloneqq \lan f, M_{\xi}T_x g\ran = \cF(f \cdot \overline{T_x g})(\xi) = (2\pi)^{-d/2} \ird e^{-i\xi \cdot y} f(y)\overline{g(y-x)} \dd{y}. 
	\end{equation}
	Whenever concerned with functions on even-dimensional spaces, it is natural to consider as well the symplectic Gabor transform: if $g \in \cS(\rdd)\smo$, 
	\begin{equation}
		\cV_g f(z,\zeta) \coloneqq \cFs (f \cdot \overline{T_z g})(\zeta) = \pi^{-d} \irdd e^{-2 i \sigma(\zeta,w)}f(w)\overline{g(w-z)} \dd{w} = 2^d V_g f(z,2J\zeta).  
	\end{equation}
	
	The phase space summability of the Gabor transform gives rise to modulation spaces. Given $1 \le p,q \le \infty$ and $g \in \cS(\rd)$, the modulation space $M^{p,q}(\rd)$ consists of all the distributions $f \in \cS'(\rd)$ such that
	\begin{equation}\label{eq-modsp-norm}
		\| f \|_{M^{p,q}_{\text{std}}} \coloneqq \Big( \ird \Big( \ird |V_g f(x,\xi)|^p \dd{x} \Big)^{q/p} \dd{\xi} \Big)^{1/q} < \infty,
	\end{equation} with obvious modifications in the case where $p=\infty$ or $q=\infty$. We write $M^p(\rd)$ if $q=p$. 
	
	Modulation spaces are a family of Banach spaces such that $\cS(\rd)$ is a dense subset of any $M^{p,q}$ with $\max\{p,q\}<\infty$, and using different windows to compute the Gabor transform results into equivalent norms. Moreover, we have the inclusions $M^{p_1,q_1}(\rd) \subseteq M^{p_2,q_2}(\rd)$ if $p_1\le p_2$ and $q_1 \le q_2$, as well as the identification $M^2(\rd) = L^2(\rd)$ (with equivalent norms).
	
	Swapping the integration order in \eqref{eq-modsp-norm} gives rise to a different family of spaces. Precisely, $W^{p,q}(\rd)$ consists of $f \in \cS'(\rd)$ such that
	\begin{equation}\label{eq:amalgam-std-norm}
		\| f \|_{W^{p,q}_{\text{std}}} \coloneqq \Big( \ird \Big( \ird |V_g f(x,\xi)|^p \dd{\xi} \Big)^{q/p} \dd{x} \Big)^{1/q} < \infty. 
	\end{equation} 
	The connection between the latter and modulation spaces is given by the Fourier transform: in view of the pointwise identity \begin{equation}\label{eq-stft-fou}
		V_{\hat{g}}\hat{f}(\xi,-x) = e^{i x\cdot \xi} V_g f(x,\xi),
	\end{equation} we have
	\begin{equation}
		f \in M^{p,q}(\rd) \Longleftrightarrow \hat{f} \in W^{p,q}(\rd). 
	\end{equation} Note that if $q=p$ then $M^p(\rd) = W^p(\rd)$, and these are Fourier-invariant spaces. 
	
	It is worth emphasizing that the spaces $W^{p,q}(\rd)$ considered here can be framed into the general theory of Wiener amalgam spaces, introduced by Feichtinger in \cite{fei_81}. In particular, from \eqref{eq:amalgam-std-norm} we note that $\| f \|_{W^{p,q}_{\text{std}}} \asymp \Big( \ird \|f \cdot \overline{T_x g}\|_{\cF L^p}^q \dd x \Big)^{1/q}$, hence $W^{p,q}(\rd)$ coincides with the amalgam space $W(\cF L^p,L^q)(\rd)$ of distributions that locally have the same regularity of the Fourier transform of a $L^p$ function and enjoy global $L^q$ decay. An equivalent discrete characterization resorting to bounded uniform partitions of unity can be provided as well, see again \cite{fei_81} for further details.
	
	Useful embeddings for standard Lebesgue and Fourier-Lebesgue spaces are known \cite{sugimoto}: if $p'$ denotes the H\"older conjugate index associated with $p \ge 1$ (namely, $1/p+1/p'=1$), then
	\begin{equation}\label{eq-embed-lp-amalg}
		W^{p_1,p}(\rd) \hookrightarrow L^p (\rd) \hookrightarrow W^{p_2,p}(\rd), 
	\end{equation} for all $p_1 \le \min\{p,p'\}$ and $p_2 \ge \max\{p,p'\}$. 
	
	In the case of functions and distributions on $\rdd$ it is natural to introduce symplectic versions of modulation and amalgam norms. By resorting to the symplectic Gabor transform, we consider
	\begin{equation}
		\| f \|_{M^{p,q}_{\text{sym}}} \coloneqq \Big( \irdd \Big( \irdd |\cV_g f(z,\zeta)|^p \dd{z} \Big)^{q/p} \dd{\zeta} \Big)^{1/q},
	\end{equation}
	\begin{equation}
		\| f \|_{W^{p,q}_{\text{sym}}} \coloneqq \Big( \irdd \Big( \irdd |\cV_g f(z,\zeta)|^p \dd{\zeta} \Big)^{q/p} \dd{z} \Big)^{1/q}. 
	\end{equation}
	It is easy to realize that such norms are equivalent to those already introduced on $M^{p,q}(\rdd)$ and $W^{p,q}(\rdd)$ respectively. In order to lighten the notation below, we assume the following convention hereinafter:  if $f \in \cS'(\bR^n)$, then we set
	\begin{equation}
		\|f\|_{M^{p,q}} \coloneqq \begin{cases} \|f\|_{M^{p,q}_{\text{sym}}} & (n \text{ even}) \\
			\|f\|_{M^{p,q}_{\text{std}}} & (n \text{ odd}), \end{cases}  \qquad \|f\|_{W^{p,q}} \coloneqq \begin{cases} \|f\|_{W^{p,q}_{\text{sym}}} & (n \text{ even}) \\
			\|f\|_{W^{p,q}_{\text{std}}} & (n \text{ odd}). \end{cases}
	\end{equation}
	
	We emphasize that, in view of the identity $\cV_{g_\sigma} f_\sigma (\zeta,z)=e^{2i \sigma(z,\zeta)}\cV_g f(z,\zeta)$, we also have the characterization $W^{p,q}(\rdd) = \cFs M^{p,q}(\rdd)$, that is
	\begin{equation}\label{eq-modwiesympfou}
		f \in M^{p,q}(\rdd) \Longleftrightarrow f_\sigma \in W^{p,q}(\rdd). 
	\end{equation}
	
	The phase space summability of a function can be further refined using suitable weight functions. For our purposes, it is enough to consider the family of polynomial weights $v_s(y) \coloneqq \lan y \ran^s$, for $s \in \bR$ and $y \in \rd$, where $\lan y \ran \coloneqq (1+|y|^2)^{1/2} \asymp 1+|y|$. The modulation space $M^{p,q}_{s}(\rd)$ is defined as before, the (standard) norm being 
	\begin{equation}
		\| f \|_{M^{p,q}_{s}} \coloneqq \Big( \ird \Big( \ird |V_g f(x,\xi)|^p v_s((x,\xi))^p \dd{x} \Big)^{q/p} \dd{\xi} \Big)^{1/q}. 
	\end{equation} 
	Similar definitions are given for $W^{p,q}_{s}(\rd)$ and the corresponding symplectic versions for functions on $\rdd$ --- along with the same convention on the norms agreed in the unweighted case. 
	
	To avoid separate discussions and workarounds due to the lack of density of the Schwartz class in $M^{p,q}$ and $W^{p,q}$ with $p=\infty$ or $q=\infty$, we agree that an operator $T \colon \cS(\rn) \to \cS'(\rn)$ is said to be bounded $\X1_{s_1}(\rn) \to \X2_{s_2}(\rn)$ if there exists a constant $C>0$ such that
	\begin{equation}
		\| Tf \|_{\X2_{s_2}} \le C \|f\|_{\X1_{s_1}}, \qquad \forall \, f \in \cS(\rn), 
	\end{equation} where $X^{p,q}_s$ denotes either $M^{p,q}_s$ or $W^{p,q}_s$ for every possible choice of $1 \le p,q \le \infty$ and $s \in \bR$. 
	
	Recall that the Schwartz kernel theorem states that a linear operator $T \colon \cS(\rn) \to \cS'(\rn)$ is continuous if and only if there exists a (unique) distribution $K_T \in \cS'(\rnn)$ such that 
	\begin{equation}
		\lan T f,g \ran = \lan K_T, g \otimes \overline{f} \ran, \qquad f,g \in \cS(\rn).
	\end{equation} Results in the same vein have been proved for modulation spaces, and we isolate from \cite[Corollary 8]{balazs} and \cite[Theorem 3.6]{CN_kernel} a characterization of boundedness that plays a key role below. 
	
	\begin{proposition}\label{prop-kernelt}
		Let $T \colon \cS(\rdd) \to \cS'(\rdd)$ be a linear operator with kernel $K_T \in \cS(\rn)$. Given $a,b \in \bR$: 
		\begin{enumerate}[label=(\roman*)]\setlength\itemsep{0.5cm}
			\item The operator $T$ is bounded $M^1_a(\rn) \to M^1_{b}(\rn)$ if and only if, for some (hence any) $G \in \cS(\rnn)\smo$, 
			\begin{equation}\label{eq-ker-I1}
				I_1 \coloneqq \sup_{(u,v) \in \rn \times \rn} \lan (u,v) \ran^{b} \int_{\rn \times \rn} |V_G K_T(z,u,w,-v)| \lan(z,w)\ran^{-a}\dd{z} \dd{w} < \infty, 
			\end{equation} and in this case we have $\|T\|_{M^1_a \to M^1_{b}} \asymp I_1$. 
			\item The operator $T$ is bounded $M^\infty_a(\rn) \to M^\infty_{b}(\rn)$ if and only if, for some (hence any) $G \in \cS(\rnn)\smo$, 
			\begin{equation}\label{eq-ker-Iinf}
				I_\infty \coloneqq \sup_{(u,v) \in \rn \times \rn} \lan (u,v) \ran^{b} \int_{\rn \times \rn} |V_G K_T(u,z,v,-w)| \lan(z,w)\ran^{-a}\dd{z} \dd{w} < \infty,
			\end{equation} and in this case we have  $\|T\|_{M^\infty_a \to M^\infty_{b}} \asymp I_\infty$.	
		\end{enumerate}
	\end{proposition}
	
	Another result that is needed below concerns complex interpolation of modulation spaces, which behaves as detailed in the following result --- see \cite[Proposition 2.3.17]{CR_book}. 
	
	\begin{proposition}\label{prop-modsp-interpol}
		Consider $1 \le p_1,p_2,q_1,q_2 \le \infty$ with $q_2<\infty$ and $r_1,r_2,s_1,s_2 \in \bR$. Let $T$ be a linear operator and  $A_1,A_2>0$ constants such that
		\begin{equation}
			\|T f \|_{M^{p_1,q_1}_{s_1}} \le A_1 \|f\|_{M^{p_1,q_1}_{r_1}}, \qquad \forall \, f \in M^{p_1,q_1}_{r_1}(\rd),
		\end{equation} 
		\begin{equation}
			\|T f \|_{M^{p_2,q_2}_{s_2}} \le A_2 \|f\|_{M^{p_2,q_2}_{r_2}}, \qquad \forall \, f \in M^{p_2,q_2}_{r_2}(\rd). 
		\end{equation}
		Then, for all $0<\theta<1$ and $1 \le p, q \le \infty$, $r,s \in \bR$ such that 
		\begin{equation}
			\frac{1}{p} = \frac{1-\theta}{p_1}+\frac{\theta}{p_2}, \quad \frac{1}{q} = \frac{1-\theta}{q_1}+\frac{\theta}{q_2}, \quad r=r_1(1-\theta)+r_2\theta, \quad s=s_1(1-\theta)+s_2\theta,
		\end{equation} there exists a constant $C>0$ independent of $T$ such that 
		\begin{equation}
			\|T f \|_{M^{p,q}_{s}} \le CA_1^{1-\theta}A_2^\theta \|f\|_{M^{p,q}_{r}}, \qquad \forall \, f \in M^{p,q}_r(\rd). 
		\end{equation}
	\end{proposition}
	
	We conclude this brief review of tools from time-frequency analysis by mentioning some alternative (quadratic) phase space representations. In particular, the ambiguity transform of $f,g \in \cS(\rd)$ is defined by
	\begin{equation}
		A(f,g)(x,\xi) \coloneqq (2\pi)^{-d} \ird e^{-i\xi \cdot y} f(y+x/2) \overline{g(y-x/2)} \dd{y},
	\end{equation}
	while the so-called Wigner transform is given by
	\begin{equation}
		W(f,g)(x,\xi) \coloneqq (2\pi)^{-d} \ird e^{-i\xi \cdot y} f(x+y/2) \overline{g(x-y/2)} \dd{y}.
	\end{equation}
	We write $A(f)$ and $W(f)$ if $g=f$. It is useful to highlight the relations among the different transforms. Setting $z=(x,\xi)\in \rdd$, we have 
	\begin{equation}
		A(f,g)(x,\xi)= e^{i\frac{x\cdot \xi}{2}} V_g f(x,\xi), \qquad A(f,g)(z)=2^{-d} \cFs W(f,g)(z/2) = \cF W(f,g)(Jz). 
	\end{equation}
	
	The following continuity result will be used below --- see \cite{CN_sharpint} for additional details.
	\begin{proposition}\label{prop-amb-boundwpq}
		Let $1 \le p_0,q_0,p_1,q_1,p_2,q_2 \le \infty$ be indices such that
		\begin{equation}
			q_2 \ge \max\{p_0,q_0,p_1,q_1\}, \qquad \min\Big\{\frac{1}{p_0}+\frac{1}{p_1},\frac{1}{p_0}+\frac{1}{p_1} \Big\}\ge\frac{1}{p_2}+\frac{1}{q_2}.
		\end{equation}
		Then, for all $f \in \M0(\rd)$ and $g \in \M1(\rd)$, 
		\begin{equation}
			\|A(f,g)\|_{\W2} \asymp \|W(f,g)\|_{\M2} \lesssim \|f\|_{\M0}\|g\|_{\M1}. 
		\end{equation}
	\end{proposition}
	
	\subsection{Weyl product and twisted convolution}\label{sec-weyl-twist}
	The Weyl pseudo-differential operator $a^\w\colon \cS(\rd) \to \cS'(\rd)$ with symbol $a \in \cS'(\rdd)$ is defined by the rule
	\begin{equation}
		\lan a^\w f,g \ran = \lan a, W(g,f) \ran, \qquad f,g, \in \cS(\rd). 
	\end{equation}
	At a formal level, we thus have the explicit representation
	\begin{equation}
		a^\w f(x)= (2\pi)^{-d} \irdd e^{i\xi \cdot (x-y)} a\Big(\frac{x+y}{2},\xi \Big)f(y) \dd{y}\dd{\xi}. 
	\end{equation}
	
	Continuity results for Weyl operators depend on the target spaces as well as on the symbols. Consider for instance the space $S^{m}(\rdd)$, $m \in \bR$, of smooth functions $a \colon \rdd \to \bC$ such that for every $\a \in \bN^{2d}$ there exists $C_\a>0$ such that
	\begin{equation}
		|\partial^\a a(z)| \le C_\a  \lan z \ran^m, \quad z \in \rdd. 
	\end{equation} Note that there is no gain in decay upon differentiation for symbols in $S^m$. On the other hand, the condition
	\begin{equation}
		|\partial^\a a(z)| \le C_\a \lan z \ran^{m-|\a|}, \quad z \in \rdd,
	\end{equation} characterizes the Shubin symbol class $\Gamma^m(\rdd)$. The classes $\Gamma^m$ and $S^m$ become Fréchet spaces when the natural seminorms associated with the corresponding definitions are taken into account, and clearly $\Gamma^m \subset S^m$. The time-frequency analysis of smooth symbol classes performed in \cite{bastianoni,toft} (see also \cite[Theorem 2.2]{bhimani_21}) yields the following result. 
	\begin{proposition}\label{prop-weyl-sm}
		Consider $m,s \in \bR$ and $1 \le p,q \le \infty$. The Weyl operator $a^\w \colon \cS(\rd) \to \cS'(\rd)$ with symbol $a \in S^m(\rdd)$ extends to a bounded operator $M^{p,q}_{s+m}(\rd) \to M^{p,q}_s(\rd)$, the operator norm depending only on a finite number of seminorms of $a$ in $S^m$. 
	\end{proposition}
	
	The composition of Weyl operators associates with a bilinear form on symbols, usually known as the Weyl (or twisted) product: we have $a^\w b^\w = (a\#b)^\w$, where (formally)
	\begin{equation}
		a\#b(z)=(4\pi)^{-2d} \int_{\bR^{4d}} e^{\frac{i}{2} \sigma(u,v)} a(z+u/2)b(z-v/2) \dd{u} \dd{v}.
	\end{equation}
	Alternatively, the Weyl product can be recast in terms of twisted convolution. The twisted convolution of $a,b \in \cS(\rdd)$ is defined by 
	\begin{equation}\label{eq-def-twistconv}
		a \times b(z) \coloneqq  \pi^{-d} \irdd e^{2i\sigma(z,w)} a(z-w)b(w) \dd{w} = \pi^{-d} \irdd e^{-2i\sigma(z,w)} a(w)b(z-w) \dd{w}. 
	\end{equation}
	Note that $a\times b (z) = \cFs(a (T_z b)^\vee)(z)$. The twisted convolution extends to a continuous (non-commutative) multiplication on $L^p(\rdd)$ for $1\le p \le 2$, as well as $\cS'(\rdd)\times \cS(\rdd) \to \cS'(\rdd)$. 
	
	As anticipated, twisted convolutions are intimately related to products of Weyl symbols: it is easy to prove that
	\begin{equation}\label{eq-def-twp}
		a \# b = a_\sigma \times b, \qquad a,b \in \cS(\rdd). 
	\end{equation}
	In particular, since the twisted convolution satisfies 
	\begin{equation} \cFs(a \times b) = a \times b_\sigma = a_\sigma \times b^\vee, \end{equation} the symplectic Fourier transform turns the twisted product into the twisted convolution:
	\begin{equation}\label{eq-symfou-intertw}
		\cFs(a \# b) = a_\sigma \times b_\sigma = a \# b_\sigma = a \times b^\vee = a_\sigma \# b^\vee. 
	\end{equation}
	
	The Weyl product extends to a continuous map on modulation spaces, in accordance with the following result --- see \cite{CTW} for further details. 
	
	\begin{proposition}\label{prop-cwt}
		Consider $1\le p_0,q_0,p_1,q_1,p_2,q_2 \le \infty$. The Weyl product $\# \colon \cS(\rdd) \times \cS(\rdd) \to \cS(\rdd)$ extends to a continuous operator $\M0(\rdd) \times \M1(\rdd) \to \M2(\rdd)$ if and only if	
		\begin{equation}\label{eq-condsharpcont}
			\max \Big\{1+\frac{1}{q_2}-\frac{1}{q_0}-\frac{1}{q_1},0\Big\} \le \min \Big\{\frac1{p_0},\frac1{p_1},\frac1{p_2'},\frac1{q_0'}, \frac1{q_1'}, \frac1{q_2}, \frac{1}{p_0}+\frac{1}{p_1}-\frac{1}{p_2} \Big\}. \end{equation}
	\end{proposition}
	The condition in \eqref{eq-condsharpcont} is equivalent to \begin{equation}\label{}
		\frac1{p_2} \le \frac1{p_1} + \frac1{p_0}, \qquad q_2 \ge q_0,q_1, \qquad 1 \le \frac1{q_1} + \frac1{q_0},
	\end{equation}
	\begin{equation}\label{}
		\max\Big\{ 1-\frac1{p_1} +\frac1{q_2}, 1-\frac1{p_0}+\frac1{q_2},\frac1{p_2} + \frac1{q_2},1+\frac1{q_2} \Big\} \le \frac1{q_1} + \frac1{q_0}.
	\end{equation}
	In particular, let us stress that boundedness of the Weyl product does not hold if $q_2<q_0$ or $q_2<q_1$. A straightforward proof of this fact is given in \cite[Proposition 3.3]{holst}. Actually, inspecting the argument given there reveals that a stronger conclusion can be drawn, as detailed below --- we retrace here the steps of the proof for the sake of clarity. 
	
	\begin{proposition}\label{prop-sharpiff}
		Given $1 \le p_1,q_1,p_2,q_2 \le \infty$ such that $q_2 < q_1$ and $b \in M^1(\rdd)\smo$, the mapping $\#_b \colon M^1(\rdd) \to M^1(\rdd)$ defined by $\#_b(a) = a\#b$ does not extend to a continuous map $\M1(\rdd) \to \M2(\rdd)$. The same result holds as well for the map $a \mapsto b \#a$. 
	\end{proposition}
	\begin{proof}
		Recall that modulation spaces are increasing with the indices, hence there is no loss of generality if one considers the case where $p_1=1$ and $q_2 \le p_2$. We argue by contradiction, assuming that the map $\#_b$ is continuous $\M1(\rdd) \to \M2(\rdd)$. Let $h \in M^1(\rd)$ be such that $\lan b^\w h,h \ran = \lan b, W(h) \ran \ne 0$. Fix a non-trivial $f \in M^{1,q_1}(\rd)\setminus \M2(\rd)$ and set $a=W(f,h)$. By Proposition \ref{prop-amb-boundwpq} we have that $a \in M^{1,q_1}(\rdd)$, hence the assumption on the continuity of $\#_b$ implies that $a\#b \in \M2(\rdd)$. In particular, by \cite[Theorem 4.4.15]{CR_book} we have that $(a\#b)^\w$ is bounded $M^{p_2'}(\rd) \to M^{q_2}(\rd)$. Given the structure of the symbol $a$, explicit computations show that the action of $a^\w$ is given by
		\begin{equation}
			a^\w g = (W(f,h))^\w g = (2\pi)^{-d/2} \lan g,h \ran f, \quad g \in M^1(\rd). 
		\end{equation} In particular, we have
		\begin{equation}
			(a\#b)^\w h = a^\w (b^\w h) = (2\pi)^{-d/2} \lan b, W(h) \ran f \in M^{1,q_1}(\rd) \setminus \M2(\rd).
		\end{equation}
		Since $M^{q_2}(\rd) \subseteq \M2(\rd)$, this shows that $(a\#b)^\w$ is not continuous $M^1(\rd) \to M^{q_2}(\rd)$, therefore not even $M^{p_2'}(\rd) \to M^{q_2}(\rd)$, which is a contradiction. 
	\end{proof}
	
	\begin{remark}\label{rem-modwp} 
		In view of the relationships  \eqref{eq-modwiesympfou} and \eqref{eq-symfou-intertw}, Proposition \ref{prop-cwt} reflects into a characterization (under the same assumptions) of the boundedness of twisted convolutions on Wiener amalgam spaces. Furthermore, again as a consequence of \eqref{eq-symfou-intertw}, we have 
		\[ \|a \times b \|_{\M2} = \| a_\sigma \# b \|_{\M2} = \| a \# b_\sigma^\vee \|_{\M2}, \] hence the same characterization applies to the boundedness of the twisted convolution as a map $\W1 \times \M0 \to \M2$  as well as $\M1 \times \W0 \to \M2$. 
		
		\noindent In particular, Proposition \ref{prop-sharpiff} implies that if $b \in W^1(\rd)=M^1(\rd)$ is a non-trivial given function, the maps $a \mapsto a \times b$ and $a \mapsto b \times a$ fail to be continuous $\M1 (\rdd) \to \M2(\rdd)$ in the case where $q_2 < q_1$.
	\end{remark}
	
	\subsection{Review of Hermite and Laguerre functions}\label{sec-hermite-laguerre}
	
	Consider the Hermite operator $\cH=-\Delta + |x|^2$ on $\rd$. It is well known that this is a positive definite self-adjoint operator on $L^2(\rd)$ with domain $D(\cH)=\{f \in L^2(\rd) : \cH f \in L^2(\rd)\}$. The corresponding spectrum $\sigma(\cH)$ is discrete, precisely
	\begin{equation}
		\sigma(\cH) = d+2\bN = \{ d+2k : k = 0,1,2, \ldots \}. 
	\end{equation} 
	In order to describe the structure of the eigenspaces, we introduce the Hermite functions. Recall that the Hermite polynomial of degree $n$ on $\bR$ is defined by
	\begin{equation}
		H_n(y) \coloneqq (-1)^n e^{y^2}\frac{\dd^n}{\dd{y}^n}(e^{-y^2}).
	\end{equation} The normalized Hermite functions are then defined by
	\begin{equation}
		h_n(y) \coloneqq (2^n n! \sqrt{\pi})^{-1/2} e^{-y^2/2} H_n(y). 
	\end{equation} 
	These are bounded functions, uniformly in $n$, and form an orthonormal basis of $L^2(\bR)$. 
	
	The Hermite functions on $\rd$ are obtained by taking elementary tensors: for a multi-index $\alpha \in \bN^d$, we set 
	\begin{equation}
		\Phi_\alpha(x) \coloneqq h_{\alpha_1}(x_1)\cdots h_{\alpha_d}(x_d) = \frac{(-1)^{|\alpha|}}{(2^{|\alpha|} \alpha! \pi^{d/2})^{1/2}} e^{|x|^2/2} \partial_x^\alpha e^{-|x|^2}, \qquad x \in \rd,
	\end{equation} where we set $|\a| \coloneqq \a_1 + \ldots + \a_d$. It is clear that $\{ \Phi_\alpha : \alpha \in \bN^d \}$ is an orthonormal basis of $L^2(\rd)$ (as well as an unconditional Schauder basis of $\cS(\rd)$), actually consisting of eigenfunctions of the harmonic oscillator:
	\begin{equation}
		\cH \Phi_\alpha = (d+2|\alpha|)\Phi_\alpha. 
	\end{equation}
	Therefore, the eigenspace associated with the eigenvalue $d+2k$ is spanned by $\{\Phi_\alpha : |\alpha|=k\}$, hence it is finite dimensional. In particular, the orthogonal projection $P_k = \sum_{|\a|=k} \lan \cdot,\Phi_\a\ran \Phi_\a$ is regularizing --- that is, $P_k \colon \cS'(\rd) \to \cS(\rd)$ is continuous. 
	
	Recall that the spectral multiplier $m(\cH)$ associated with a Borel measurable function $m \colon \sigma(\cH) \to \bC$ is defined by
	\begin{equation}
		m(\cH)f \coloneqq \sum_{k \in \bN} m(d+2k)P_k f, \qquad f \in D(m(\cH)),
	\end{equation} with maximal domain $D(m(\cH))= \Big\{f \in L^2(\rd) : \sum_{k \in \bN} |m(d+2k)|^2 \|P_k f\|_{L^2}^2 < \infty \Big\}$. 
	
	Let us consider now the special Hermite operator on $\rdd$, also known as the twisted Laplacian. Recall from the Introduction that it is defined as the sum of the Hermite operator with a rotation operator --- to be precise, setting $z=(x,y) \in \rdd$, we have 
	\begin{equation}
		\cL \coloneqq - \sum_{j=1}^d \Big[\big(\partial_{x_j}-\frac{i}{2} y_j \big)^2+\big(\partial_{y_j}+\frac{i}{2}x_j\big)^2 \Big] = -\Delta_z + \frac{1}{4}|z|^2 -i\sum_{j=1}^d  \big(x_j\partial_{y_j} - \partial_{x_j}y_j\big).
	\end{equation}
	
	As far as the spectral features of $\cL$ are concerned, the theory of Landau-Weyl operators \cite{degoss_09,degoss_10} implies that some of the properties of $\cH$ extend to $\cL$. In particular, $\cH$ and $\cL$ have the same eigenvalues (hence $\sigma(\cL)= d+2\bN$), but these are now infinitely degenerate. More precisely, each eigenfunction $h$ of $\cH$ associates with a family of eigenfunctions of $\cL$ of the form $\{ U_{g}h : g \in \cS(\rd), \|g\|_{L^2}=1 \}$, where $U_g$ are the partial isometries $L^2(\rd) \to L^2(\rdd)$ given by ambiguity transforms indexed by normalized Schwartz functions:
	\begin{equation}
		U_g h (x,\xi) \coloneqq  A(h,g)(x,\xi).  
	\end{equation} In fact, these results follow from the more general intertwining relation $\cL U_g = U_g \cH$. 
	
	The previous results naturally lead us to introduce the so-called special Hermite functions, indexed by $\alpha,\beta \in \bN^d$:
	\begin{equation}
		\Phi_{\a,\b}(z)\coloneqq A(\Phi_\a,\Phi_\b)(Jz) = \cF W(\Phi_\a,\Phi_b)(-z), \quad z \in \rdd.
	\end{equation} 
	Exploiting the connection between the ambiguity and the Gabor transforms, as well as the identity \eqref{eq-stft-fou}, it is not difficult to show that $A(\Phi_a,\Phi_b)(Jz) = A(\widehat{\Phi_\a},\widehat{\Phi_\b})(z)$. It is then useful to recall that the Hermite functions are eigenfunctions for the Fourier transform, namely $\widehat{\Phi_\a} = (-i)^{|\a|}\Phi_\a$, hence
	\begin{equation}
		|\Phi_{\a,\b}(z)| = |A(\Phi_\a,\Phi_\b)(z)|, \quad z \in \rdd.
	\end{equation}
	The family $\{ \Phi_{\a,\b} : \a,\b \in \bN^d \}$ is clearly an orthonormal basis of $L^2(\rdd)$ consisting of eigenfunctions of $\cL$ --- as well as of the harmonic oscillator on $L^2(\rdd)$:
	\begin{equation}
		\cL\Phi_{\a,\b}= (d+2|\b|)\Phi_{\a,\b}, \qquad \cH\Phi_{\a,\b} = (d+|\a|+|\b|)\Phi_{\a,\b}.
	\end{equation}
	Moreover, the eigenspace of $\cL$ associated with the eigenvalue $d+2k$ is spanned by $\{ \Phi_{\a,\b} : |\b|=k\}$. The corresponding eigenprojection $Q_k$ can be given a remarkable form in light of some properties satisfied by the special Hermite functions. In particular, straightforward computations show that
	\begin{equation}
		\Phi_{\a,\b} \times \Phi_{\mu,\nu} =  4^d \Phi_{\a,\nu} \delta_{\b,\mu}. 
	\end{equation}
	Therefore, in light of the expansion $f = \sum_{\a,\b \in \bN^d} \lan f, \Phi_{\a,\b} \ran \Phi_{\a,\b \in \bN^d}$, we have
	\begin{equation}
		f \times \Phi_{\nu,\nu} = \sum_{\a,\b \in \bN^d} \lan f, \Phi_{\a,\b} \ran \Phi_{\a,\b} \times \Phi_{\nu,\nu} = 4^d \sum_{\a \in \bN^d} \lan f, \Phi_{\a,\nu} \ran \Phi_{\a,\nu}. 
	\end{equation}
	As a result, we obtain
	\begin{equation}
		Q_k f  = \sum_{|\b|=k} \sum_{\a \in \bN^d}  \lan f,\Phi_{\a,\b} \ran \Phi_{\a,\b} = f \times \Big(4^{-d} \sum_{|\b|=k} \Phi_{\b,\b} \Big).
	\end{equation}
	
	It is well known that the special Hermite functions can be expressed in terms of Laguerre functions. Recall that the Laguerre polynomials $L_n^\delta$ of type $\delta>-1$ and degree $n\in \bN$ are defined by
	\begin{equation}
		L_n^\delta(y) e^{-y} y^\delta = \frac{1}{n!} \frac{\dd^n}{\dd y^n} \big(e^{-y}y^{n+\delta}\big), \qquad y \ge 0. 
	\end{equation} They satisfy the following generating function identity: for all $r \in \bR$ with $|r|<1$, 
	\begin{equation}\label{eq-laguerre-genfun}
		\sum_{n=0}^{\infty} L_n^\delta(y)r^n = (1-r)^{-\delta-1} e^{-\frac{r}{1-r}y}. 
	\end{equation}
	Combining this identity with the Fourier transform and the definition of the special Hermite functions yields 
	\begin{equation}\label{eq-phibb-def}
		\Phi_{\b,\b}(x,\xi) = (2\pi)^{-d} \prod_{j=1}^{d} L_{\beta_j}^0\Big(\frac{|x_j|^2+|\xi_j|^2}{2}\Big) \exp\Big( -\frac{|x_j|^2+|\xi_j|^2}{4}\Big). 
	\end{equation}
	A straightforward comparison of the associated generating functions shows that
	\begin{equation}
		(2\pi)^d\sum_{|\b|=k} \Phi_{\b,\b}(z) =  L_k^{d-1}\Big(\frac{|z|^2}{2}\Big) e^{-\frac14 |z|^2} \eqqcolon \vp_k(z). 
	\end{equation} 
	As a consequence, the projection $Q_k$ has an explicit structure as twisted convolution:
	\begin{equation}\label{eq-Qk-def}
		Q_k f = (8\pi)^{-d} f \times \vp_k.
	\end{equation} 
	
	The spectral multiplier $m(\cL)$ associated with a Borel measurable function $m \colon \sigma(\cH) \to \bC$ is then defined by
	\begin{equation}
		m(\cL)f \coloneqq \sum_{k \in \bN} m(d+2k)Q_k f, \qquad f \in D(m(\cL)),
	\end{equation} with maximal domain $D(m(\cL))= \Big\{f \in L^2(\rd) : \sum_{k \in \bN} |m(d+2k)|^2 \|Q_k f\|_{L^2}^2 < \infty \Big\}$. 
	
	A refined time-frequency analysis of Hermite and special Hermite functions appears to be quite difficult. Nevertheless, the following (quite pessimistic) asymptotic bounds will suffice for the purposes of this note. 
	\begin{proposition}\label{prop-polybound} Let $\beta \in \bN^d$ be such that $|\b|=k$. There exist $N \in \bN$ such that, if $k \ge N$,
		\begin{equation}\label{eq-phib-mpq}
			\|\Phi_\b\|_{M^1} \asymp k^{d^2-d/2}, \quad  
		\end{equation}
		\begin{equation}\label{eq-phibb-wpq}
			\| \Phi_{\b,\b} \|_{M^1} \lesssim k^{2d^2-d},
		\end{equation}
		\begin{equation} \label{eq-vpk-wpq}
			\|\vp_k\|_{M^1} \lesssim k^{2d^2-1}.
		\end{equation}
	\end{proposition}
	\begin{proof}
		Let us first consider \eqref{eq-phib-mpq}. Since the identity $|V_g f|=|A(f,g)|$ holds pointwise, choosing $f = g =\Phi_\b \in \cS(\rdd)$ we have
		\begin{equation}
			\|\Phi_\b\|_{M^1} = \|A(\Phi_\b,\Phi_\b)\|_{L^1} = \|\Phi_{\b,\b}\|_{L^1}. 
		\end{equation}
		In light of \eqref{eq-phibb-def}, after setting $z_j = (x_j,\xi_j) \in \rdd$ the problem boils down to prove bounds for
		\begin{equation}
			\|\Phi_\b\|_{M^1} = (2\pi)^{-d} \prod_{j=1}^{d} \Big\| L_{\beta_j}^0\Big(\frac{|z_j|^2}{2}\Big) \exp\Big( -\frac{|z_j|^2}{4}\Big) \Big\|_{L^1_{z_j}}.
		\end{equation}
		Sharp asymptotic $L^1$ bounds for Laguerre functions are known, see \cite[Lemma 1.5.4]{thangavelu_93}, and straightforward computations lead to the claim.
		
		Concerning \eqref{eq-phibb-wpq}, application of Proposition \ref{prop-amb-boundwpq} with $p_2=q_2=1$ yields
		\begin{equation}
			\| \Phi_{\b,\b} \|_{M^1} \lesssim \|\Phi_\b\|_{M^1}^2, \end{equation} and the claim follows immediately from the previous step. 
		
		To conclude, the bound in \eqref{eq-vpk-wpq} follows after noticing that 
		\begin{equation}
			\|\vp_k\|_{M^1} \le \sum_{|\b|=k} \|\Phi_{\b,\b}\|_{M^1}, 
		\end{equation} as a consequence of \eqref{eq-phibb-wpq} and a standard stars-and-bars argument showing that $\sum_{|\b|=k} 1 = \binom{k+d+1}{k}$ --- the latter becoming comparable to $k^{d-1}$ for $k$ sufficiently large. 
	\end{proof}
	
	\subsection{Metaplectic operators}\label{sec-metapl}
	The metaplectic representation is a faithful and strongly continuous unitary representation on $L^2(\rd)$ of the two-fold covering $\Mp$ of the symplectic group $\Sp$ --- see \cite{CR_book,degoss_11_book} for further details. In particular, every symplectic matrix $S \in \Sp$ associates with a pair of unitary operators that differ only by the sign --- we write $\mu(S)$ to denote any of them, with a slight abuse of notation. As a result, for every $S_1,S_2 \in \Sp$ we have that $\mu(S_1)\mu(S_2)$ and $\mu(S_1 S_2)$ coincide up to the sign. 
	
	The phase space analysis of metaplectic operators has been developed over the last decade. In particular, there is a fairly complete picture of boundedness on modulation spaces. We list below some results proved in \cite[Corollary 3.4]{CNT_dispdiag20}, \cite[Theorem 1]{cauli} and \cite[Theorems 3.2 and 4.6]{FS}. 
	
	\begin{proposition}\label{prop-metap-mod}
		Let $\mu(S)$ be a metaplectic operator associated with $S \in \Sp$. 
		\begin{enumerate}[label=(\roman*)]
			\item $\mu(S)$ is an automorphism of $\cS(\rd)$ and extends to a bounded operator on $M^p(\rd)$ for every $1 \le p \le \infty$. In particular, there exists $C>0$ such that, for all $f \in M^p(\rd)$, 
			\begin{equation}
				\| \mu(S)f \|_{M^p} \le C (\sigma_1(S) \cdots \sigma_d(S))^{|1/2-1/p|} \|f \|_{M^p}, 
			\end{equation} where $\sigma_1(S) \ge \ldots \ge \sigma_d(S) \ge 1$ are the $d$ largest singular values of $S$. Moreover, the following phase space dispersive estimate holds:
			\begin{equation}
				\| \mu(S)f \|_{M^\infty} \le C (\sigma_1(S) \cdots \sigma_d(S))^{-1/2} \|f\|_{M^1}, \quad f \in M^1(\rd). 
			\end{equation} 
			\item For $1 \le p,q \le \infty$ with $p\ne q$, $\mu(S)$ is bounded on $M^{p,q}(\rd)$ if and only if $S$ is an upper block triangular matrix. 
			\item For $s \in \bR$ and $1 \le p,q \le \infty$, $\mu(S)$ is bounded on $M^{p,q}_{s}(\rd)$ if and only if it is bounded on $M^{p,q}(\rd)$. 
		\end{enumerate} 
	\end{proposition}
	
	It is also well known that metaplectic operators can be characterized as the Schr\"odinger propagators associated with quadratic Hamiltonian operators. For instance, with reference to \eqref{eq-symbol-L}, we have 
	\begin{equation}
		e^{-it\cL} = e^{-itL^\w} = c\mu(L_t), 
	\end{equation} for a suitable $c \in \bC$, $|c|=1$, where
	\begin{equation}\label{eq-bL}
		L_t = e^{-2t\bL}, \quad \bL \coloneqq \begin{bmatrix} -J/2 & I \\ -I/4 & -J/2 \end{bmatrix}. 
	\end{equation}
	Explicit integral representations for $e^{-it\cL}$ can be obtained via Mehler-type formulas for metaplectic operators \cite{CR_book,degoss_11_book,horm-mehler-95} or, as customary in physics, by means of a Wick rotation $t \mapsto it$ in the heat flow setting discussed in Section \ref{sec-heat} (more precisely, via analytic continuation \cite{muller_90}), so that for a suitable $c\in \bC$ with $|c|=1$ we have 
	\begin{equation}\label{eq-schro-rep}
		e^{-it\cL} = c (f \times q_{t}), \qquad q_{t}(z) = (16\pi \sin t)^{-d} e^{\frac{i}{4} \cot(t)|z|^2},
	\end{equation} provided that $t \notin 	\mathfrak{E} \coloneqq \{ k\pi : k \in \bZ\}$. Basic dispersive estimates can be read from this representation, such as 
	\begin{equation}
		\| e^{-it\cL} f \|_{L^\infty} \lesssim (\sin t)^{-d}  \|f\|_{L^1}, \quad f \in \cS(\rdd), \quad t \notin \mathfrak{E}.
	\end{equation} It is worth comparing this result to the parallel modulation space setting in light of the results listed in Proposition \ref{prop-metap-mod}, which need access to the singular values of $L_t$. In fact, the twisted Laplacian happens to be one of those rare cases where a closed-form expression for the companion symplectic matrix $L_t$ can be readily obtained, after noticing that the powers of the Hamiltonian matrix $\bL$ satisfy the identities
	\begin{equation}
		\bL^{2k} = (-1)^{k-1}\bL^2, \qquad \bL^{2k+1} = (-1)^k \bL, \qquad k \in \bN\smo.  
	\end{equation} We leave the explicit computations to the willing reader, and just claim that $L_t$ comes with two singular values $\ell_-(t),\ell_+(t)$ of multiplicity $2d$ each, satisfying $\ell_-(t)\ell_+(t)=1$ and $\ell_+(t) \asymp 1$ uniformly with respect to $t \in \bR$. Therefore, from Proposition \ref{prop-metap-mod} we obtain
	\begin{equation}
		\| e^{-it\cL} f \|_{M^\infty} \le C \|f\|_{M^1}, \quad f \in \cS(\rdd), \quad t \in \bR. 
	\end{equation} 
	A subtler analysis of the phase space evolution requires the tools developed in \cite{CNT_dispdiag20}. We also address the reader to \cite{carypis,CNR_15,pravda_18} for the study of Gabor-type wave front sets (see Remark \ref{rem-wfs} below), and the corresponding problem of propagation of singularities in phase space.
	
	\begin{remark}\label{rem-mult-reg} Let $m$ be the Fourier transform of a finite complex Borel measure $\mu$ on $\bR$, that is
		\begin{equation}
			m(\xi) = (2\pi)^{-1/2} \int_{\bR} e^{-i x \xi} \dd{\mu}(x).
		\end{equation}
		It is then straightforward to check the (pointwise) identity
		\begin{equation}
			m(\cH)f(y) = (2\pi)^{-1/2} \int_{\bR} e^{-ix\cH}f(y) \dd{\mu}(x), \qquad y \in \rd, \, f \in \cS(\rd).
		\end{equation}
		The Schr\"odinger propagator $e^{-ix\cH}$ is a metaplectic operator satisfying $\|e^{-ix\cH} f\|_{M^p} \le C \|f\|_{M^p}$, where the constant $C$ does not depend on $x$ in view of the periodic phase space dynamics of the harmonic oscillator. As a result, $m(\cH)$ is a bounded operator on $M^p(\rd)$ for all $1 \le p \le \infty$, and similar arguments apply to $m(\cL)$. 
	\end{remark} 
	
	\section{A transference principle for the spectral multipliers of $\cL$}\label{sec-transf}
	
	As anticipated in the Introduction, we first establish a unitary (in fact, metaplectic) intertwining relationship between the twisted Laplacian and a partial harmonic oscillator. This result is consistent with findings obtained in \cite{buzano,gramchev_10,gramchev_09} and will be the main ingredient of the subsequent transference principle. 
	
	\begin{proposition}
		Let $\cA_J$ be the operator defined by
		\begin{equation}\label{eq-AJ-def}
			\cA_J f (x,y) \coloneqq (2\pi)^{-d} \ird e^{i x \cdot u} f(u+y/2,u-y/2) \dd{u}, \qquad f \in \cS(\rdd). 
		\end{equation}
		\begin{enumerate}[label=(\roman*)]\setlength\itemsep{0.5cm}
			\item $\cA_J$ is a metaplectic operator, namely
			\begin{equation}
				\cA_J = c \mu(A_J), \qquad A_J = \begin{bmatrix}
					A_1 & A_2 \\ A_3 & A_4 \end{bmatrix},
			\end{equation} for some $c \in \bC$, $|c|=1$, where the $2d\times 2d$ blocks of $A_J$ are given by
			\begin{equation}
				A_1 = \begin{bmatrix} -I/2 & -I/2 \\ O & O \end{bmatrix}, \quad A_2 = \begin{bmatrix} O & O \\ I/2 & -I \end{bmatrix}, \quad A_3 = \begin{bmatrix} O & O \\ I & -I \end{bmatrix}, \quad A_4 = \begin{bmatrix} -I/2 & -I \\ O & O \end{bmatrix}.
			\end{equation}
			As such, $\cA_J$ is a unitary operator on $L^2(\rdd)$ that extends to a continuous operator on $M^p_s(\rdd)$ for all $1 \le p \le \infty$ and $s \in \bR$.
			
			\item The following intertwining relation holds:
			\begin{equation}\label{eq-AJ-intertw}
				\cA_J (I \otimes \cH) = \cL \cA_J. 
			\end{equation}
		\end{enumerate}
	\end{proposition}
	
	\begin{proof}
		We have $\cA_J = \cF_1^{-1} T f$, where $\cF_1$ denotes the partial Fourier transform of a function on $\rdd$ with respect to the first group of $d$ variables, that is
		\begin{equation}
			\cF_1 f(\xi,y) = (2\pi)^{-d/2} \ird e^{-i\xi \cdot x} f(x,y) \dd{x}, 
		\end{equation} and $T$ amounts (up to a factor) to a linear change of variables: 
		\begin{equation}
			Tf(u,v) \coloneqq (2\pi)^{-d/2} f(u+v/2,u-v/2).
		\end{equation} 
		It is well known that partial Fourier transforms and linear changes of variables are metaplectic operators --- see for instance \cite{degoss_11_book,morsche}. Boundedness on modulation spaces follows from Proposition \ref{prop-metap-mod}. The details about determining the block structure of the matrix $\cA_J$ are left to the interested reader. 
		
		Concerning the intertwining property, it is enough to show that $\cA_J(I\otimes \cH) f$ and $\cL \cA_J f$ coincide for all $f \in \cS(\rdd)$. To this aim, let us first highlight that $\cA_J$ satisfies \begin{equation}\label{eq-Aj-hermfun}
			\cA_J(\Phi_\a \otimes \Phi_\b)(x,y) = A(\Phi_\a,\Phi_\b)(y,-x) = \Phi_{\a,\b}(x,y). 
		\end{equation}
		Therefore, we expand an arbitrarily chosen $f \in \cS(\rdd)$ with respect to the orthonormal basis $\{\Phi_\a \otimes \Phi_\b : \a,\b \in \bN^d\}$ and, setting $c_{\a,\b} = \lan f, \Phi_\a\otimes \Phi_\b \ran$, we get
		\begin{align*}
			\cA_J (I \otimes \cH) f & = \sum_{\a,\b \in \bN^d} c_{\a,\b} \cA_J (I\otimes \cH) ( \Phi_\a \otimes \Phi_\b) \\
			& = \sum_{\a,\b \in \bN^d} c_{\a,\b} \cA_J ( \Phi_\a \otimes (d+2|\b|)\Phi_\b) \\
			& = \sum_{\a,\b \in \bN^d} (d+2|\b|)c_{\a,\b} \Phi_{\a,\b} \\
			& = \sum_{\a,\b \in \bN^d} c_{\a,\b} \cL \Phi_{\a,\b} \\
			& = \cL \cA_J f. \qedhere
		\end{align*}
	\end{proof}
	
	As a consequence of the previous result, we infer that $\cL$ is unitarily equivalent to $I \otimes \cH$ on $\cS(\rdd)$. In particular, they are isospectral operators with unitarily equivalent spectral projections, hence unitarily equivalent functional calculi: for all $m \in L^\infty(d+2\bN)$, we have 
	\begin{equation}\label{eq-mL-intertw}
		m(\cL) = \cA_J m(I\otimes \cH) \cA_J^*. 
	\end{equation} 
	It is straightforward to check that $\sigma(I \otimes \cH) = \sigma(\cH) = d+2\bN$, the eigenspace of $I \otimes \cH$ associated with the eigenvalue $d+2k$ being spanned by $\{ g_j \otimes \Phi_\b : j \in \bN, |\b|=k\}$, where $(g_j)_{j \in \bN}$ is any orthonormal basis of $L^2(\rd)$. Therefore, expansion with respect to the orthonormal basis $\{\Phi_\a \otimes \Phi_\b : \a,\b \in \bN^d\}$ yields the following result. 
	\begin{lemma}\label{lem-spec_tensor}
		If $m \in L^\infty(d+2\bN)$, then $m(I\otimes \cH) =I \otimes m(\cH)$. 
	\end{lemma}
	We are ready to prove a transference principle for spectral functions of the twisted Laplacian in terms of the corresponding ones for the harmonic oscillator. 
	
	\begin{theorem}\label{thm-transf}
		Let $m \in L^\infty(d+2\bN)$ and consider the spectral multipliers $m(\cL)$ on $L^2(\rdd)$ and $m(\cH)$ on $L^2(\rd)$. 
		
		If $m(\cH)$ is a bounded operator $M^p_s(\rd)\to M^p_r(\rd)$ for some $s,r \ge 0$ and every $1 \le p \le \infty$, then $m(\cL)$ is a bounded operator $M^p_s(\rdd) \to M^p(\rdd)$, $1 \le p \le \infty$, satisfying
		\begin{equation}
			\| m(\cL) \|_{M^p_s \to M^p} \le C \|m(\cH)\|^{1/p}_{M^1_s \to M^1_{r}}\|m(\cH)\|^{1/p'}_{M^\infty_s \to M^\infty_{r}}, 
		\end{equation} for some constant $C>0$. 
	\end{theorem}
	
	\begin{proof} We argue by interpolation between the extremal cases $M^1_s(\rdd) \to M^1_{}(\rdd)$ and $M^\infty_s(\rdd) \to M^\infty_{}(\rdd)$. Detailed arguments are given here for the latter case only, since the other one is identical up to obvious modifications. 
		
		With reference to Proposition \ref{prop-kernelt}, let us consider $I_\infty$ with $n=2d$, $b=0$ and $a=s \ge 0$. It is not restrictive to assume that $G=g \otimes g$ with $g=g_0\otimes g_0 \in \cS(\rdd)$ for some real-valued $g_0 \in \cS(\rd)\smo$. Combining the definition of the Gabor transform, the Schwartz kernel theorem and the identity \eqref{eq-tensor-pi}, we get
		\begin{align} \label{eq-ker-gabmatr}
			|V_G K_T(u,z,v,-w)| & = |\lan K, \pi(u,z,v,-w)g\otimes g \ran| \nonumber \\
			& = |\lan K_T, \pi(u,v)g \otimes \pi(z,-w)g \ran| \nonumber \\
			& = |\lan T \pi(z,w)g, \pi(u,v)g \ran|.
		\end{align} In the case where $T=I \otimes m(\cH)$ we have in particular
		\begin{align*}
			|V_G K_T(u,z,v,-w)|  
			& = |\lan T \pi(z,w)g, \pi(u,v)g \ran| \\
			& = |\lan (I\otimes m(\cH)) \pi(z_1,w_1)g_0 \otimes \pi(z_2,w_2)g_0, \pi(u_1,v_1)g_0 \otimes \pi(u_2,v_2)g_0 \ran| \\
			& = |\lan \pi(z_1,w_1)g_0,\pi(u_1,v_1)g_0 \ran| |\lan m(\cH) \pi(z_2,w_2)g_0, \pi(u_2,v_2)g_0 \ran| \\
			& = |V_{g_0}g_0(u_1-z_1,v_1-w_1)| |V_g K_{m(\cH)} (u_2,z_2,v_2,-w_2)|. 
		\end{align*} 
		Using that $\lan (x_1,x_2,y_1,y_2) \ran^{-q} \le \lan (x_2,y_2)\ran^{-q} \le 1$ for all $x_1,x_2,y_1,y_2 \in \rd$ and $q \ge 0$, the initial condition on the kernel of $T$ reduces to the finiteness of
		\begin{multline}
			I_\infty = \Big( \sup_{(u_1,v_1) \in \rd \times \rd} \int_{\rd \times \rd} |V_{g_0}g_0(u_1-z_1,v_1-w_1)| \dd{z_1} \dd{w_1} \Big) \\ \times \Big( \sup_{(u_2,v_2) \in \rd \times \rd} \lan (u_2,v_2) \ran^{r}\int_{\rd \times \rd} |V_g K_{m(\cH)} (u_2,z_2,v_2,-w_2)| \lan (z_2,w_2) \ran^{-s} \dd{z_2} \dd{w_2} \Big).
		\end{multline} 
		The first factor is finite since $g_0 \in \cS(\rd)$ implies $V_{g_0}g_0 \in \cS(\rd)$ --- see for instance \cite[Theorem 1.2.23]{CR_book}. Concerning the second one, from \eqref{eq-ker-gabmatr} we recognize that the finiteness of this quantity is equivalent to the condition \eqref{eq-ker-Iinf} for the kernel of the Hermite multiplier $m(\cH)$, which in turn characterizes the boundedness of the latter as an operator $M^\infty_s(\rd) \to M^\infty_{r}(\rd)$. Since this is true by assumption, we infer that the second factor is finite --- and actually equivalent to the operator norm $\|m(\cH)\|_{M^\infty_s \to M^\infty_{r}}$, so that 
		\begin{equation}
			\| I \otimes m(\cH) \|_{M^\infty_s \to M^\infty_{}} \lesssim  \|m(\cH)\|_{M^\infty_s \to M^\infty_{r}}.
		\end{equation}
		We then resort to Proposition \ref{prop-metap-mod} and Lemma \ref{lem-spec_tensor} to conclude that, for all $f \in \cS(\rdd)$, 
		\begin{align*}
			\|m(\cL)f\|_{M^\infty_{}} & = \|\cA_J m(I\otimes \cH)\cA_J^* f\|_{M^\infty_{}} \\
			& \lesssim \| (I \otimes m(\cH))\cA_J ^* f\|_{M^\infty_{}} \\
			& \le \| (I \otimes m(\cH))\|_{M^\infty_{s} \to M^\infty_{}} \|\cA_J^* f\|_{M^\infty_{s}} \\
			& \lesssim \|m(\cH)\|_{M^\infty_{s} \to M^\infty_{r}} \| f\|_{M^\infty_{s}}.
		\end{align*}
		Repeating the argument for $I_1$, we obtain similarly
		\begin{equation}
			\|m(\cL)\|_{M^1_s \to M^1} \lesssim \|m(\cH)\|_{M^1_s \to M^1_{r}}.
		\end{equation}
		
		Finally, we have proved that $m(\cL)$ is a bounded operator on both weighted $M^1(\rdd)$ and $M^\infty(\rdd)$, which is enough to obtain boundedness on all the intermediate modulation spaces $M^p(\rdd)$, $1 < p < \infty$, by complex interpolation via Proposition \ref{prop-modsp-interpol}.
	\end{proof}
	
	\begin{remark}
		In view of Proposition \ref{prop-metap-mod}, given that the symplectic matrix $A_J$ associated with the intertwining metaplectic operator $\cA_J$ is not block triangular, the transference principle in Theorem \ref{thm-transf} does not extend to $M^{p,q}$ with $p\ne q$.
		
		\noindent We emphasize that our proof exploits interpolation between characterizations of boundedness on the ``endpoint'' modulation spaces $M^1$ and $M^\infty$. Similar Schur-type characterizations at the level of the kernel in the intermediate cases $L^p\to L^q$ are known to be hard to establish --- sufficient (but non-necessary) conditions to prove boundedness $M^p \to M^q$ of an operator in terms of the regularity of its Gabor matrix can be found for instance in \cite[Proposition 3.28]{fei_22}. It is natural to wonder whether a different proof strategy could lead to a sharper transference result, such as: if $m(\cH)$ is bounded on $M^p(\rd)$, $1 < p < \infty$, then $m(\cL)$ is bounded on $M^p(\rdd)$ as well, with $
		\|m(\cL) \|_{M^p \to M^p} \lesssim \|m(\cH)\|_{M^p \to M^p}$. 
	\end{remark}
	
	\begin{remark} \label{rem-general-transf} It is easy to realize that, up to minor adjustments in the proof, the transference Theorem \ref{thm-transf} directly applies to $\cH$ and $\cL$ in place of $m(\cH)$ and $m(\cL)$ respectively. In fact, taking care of domain issues, the result also extends to more general (e.g., unbounded) spectral functions. A closer inspection of the proof of Theorem \ref{thm-transf} actually reveals that the claim still holds if $m(\cH)$ and $m(\cL)$ are replaced by linear continuous operator $T \colon \cS(\rd) \to \cS'(\rd)$ and $U(I \otimes T)V$ respectively, where $U,V$ are linear operators that are bounded on every modulation space $M^p_s(\rdd)$ with $1 \le p \le \infty$ and $s\ge 0$. 
	\end{remark}
	
	\subsection{Boundedness of $\cL$ and its fractional powers}\label{sec-fracpow}
	
	In light of the transference principle just discussed, let us investigate the boundedness of $\cL$ and its (real) fractional powers on modulation spaces. We are preliminarily led to examine the continuity of the harmonic oscillator and its powers, for which a comprehensive pseudo-differential analysis is available. 
	
	\begin{proposition}\label{prop-powerH}
		The operator $\cH^\nu$, with $\nu \in \bR$ and densely defined on $\cS(\rd)$, is pseudo-differential with Weyl symbol in the Shubin class $\Gamma^{2\nu}(\rdd)$, hence it extends to a bounded operator $M^{p,q}_{s+2\nu}(\rd)\to M^{p,q}_{s}(\rd)$ for all $1 \le p,q \le \infty$ and $s \in \bR$. 
	\end{proposition}
	\begin{proof}
		If $\nu >0$, the first part of the claim follows from the general theory of positive powers of globally elliptic Weyl operators with Shubin symbols, and is a restatement of \cite[Proposition 2.3]{bhimani_21} --- see the corresponding proof for additional details. The claim then follows by symbolic calculus \cite[Theorem 25.4]{shubin} in the case where $\nu<0$, since $\cH^\nu= (\cH^{-\nu})^{-1}$ and $0\notin \sigma(\cH)$. To conclude, boundedness on modulation spaces follows by Proposition \ref{prop-weyl-sm}, in view of the embedding $\Gamma^{2\nu} \subset S^{2\nu}$.
	\end{proof}
	
	The results in Proposition \ref{prop-powerH} motivates the analysis of general real powers of $\cL$. Given $\nu \in \bR$, we define here the fractional powers $\cL^\nu$ by spectral expansions such as
	\begin{equation} \label{eq-def-Lnu}
		\cL^\nu f \coloneqq \sum_{k \in \bN} (d+2k)^\nu Q_k f, \qquad f \in \cQ \coloneqq \text{span}\{\Phi_{\a,\b} : \a,\b \in \bN^d\}.
	\end{equation} Note that $\cQ$ is a dense subset of $\cS(\rdd)$ --- in fact, one can assume $f \in \cS(\rdd)$ as well.  
	
	If $\nu \le 0$ then boundedness of $\cL^\nu$ on $M^p(\rdd)$ follows by straightforward application of Theorem \ref{thm-transf} --- notably, at the price of the loss of the smoothing effect associated with the action of $\cH^\nu$. 
	
	The case $\nu>0$ requires additional comments, since the scope of Theorem \ref{thm-transf} is restricted to bounded multipliers. Note that the operators $\cL^\nu$ (densely defined on $\cQ$) and $(I\otimes \cH)^\nu$ (densely defined on $\cP \coloneqq \spann\{ \Phi_\a \otimes \Phi_\b : \a,\b \in \bN^d\}$) are unitarily equivalent via $\cA_J$, with $\cA_J^*(\cQ)=\cP$ as a consequence of \eqref{eq-Aj-hermfun}. One can also prove that, as in Lemma \ref{lem-spec_tensor}, $(I\otimes \cH)^\nu$ and $I\otimes \cH^\nu$ coincide on the (non-empty) intersection of their domains, for instance on $\cP$ or $\cS(\rdd)$. Transference in this setting then follows as outlined in Remark \ref{rem-general-transf}. 
	
	Let us distill the previous discussion into a boundedness result for the fractional twisted Laplacian, parallel to Proposition \ref{prop-powerH} --- note in particular the case $\nu=1$, since boundedness of $\cL$ is obtained in spite of the current lack of a refined phase space analysis of the eigenprojections $Q_k$. 
	
	\begin{corollary}\label{cor-fracpow-L}
		For every $\nu \in \bR$ and $1 \le p \le \infty$, the operator $\cL^{\nu}$ defined in \eqref{eq-def-Lnu} is a bounded operator $M^p_{\max\{2\nu,0\}}(\rdd) \to M^p(\rdd)$. 
	\end{corollary}
	
	We conclude this section with some remarks on other spectral multipliers related to powers of the Hermite and the special Hermite operators. 
	\begin{remark}
		Consider the Riesz transforms for the harmonic oscillator, namely the operators $R_j(\cH) = A_j\cH^{-1/2}$, $j=1,\ldots,d$, where $A_j=-\partial_{x_j}+x_j$ is the so-called lowering operator. They were first studied in \cite{thangavelu_90} in connection with the wave equation for $\cH$. In \cite{bhimani_19} the authors prove a boundedness result for $R_j$ on $M^{p,q}(\rd)$ under constraints on $p,q$. In fact, $R_j (\cH)$ is bounded on $M^{p,q}_s(\rd)$ for any choice of $1 \le p,q \le \infty$ and $s\in \bR$, since $\cH^{-1/2}$ is a Weyl operator with symbol in $\Gamma^{-1}$, while $A_j$ is a Weyl operator with symbol in $\Gamma^1$. As a result, $R_j(\cH)$ has symbol in $\Gamma^0\subset S^0$, hence it is bounded on every modulation space. The same result obviously holds for $\cH^{-1/2}A_j$, $j=1,\ldots,d$. 
		
		\noindent By transference, the Riesz transforms for the twisted Laplacian $R_j(\cL) = A_j \cL^{-1/2}$ are bounded on $M^p(\rd)$ for all $1 \le p \le \infty$ and $j = 1,\ldots, d$ --- see \cite{nowak_17,thangavelu_93} for the analysis of these operators in the context of Lebesgue spaces. 
	\end{remark}
	
	\subsection{Oscillating multipliers for $\cL$}
	Let us consider now oscillating multipliers of the form
	\begin{equation}
		m(x) = x^{-\delta/2} e^{ix^\gamma/2}, \quad x,\gamma,\delta>0.
	\end{equation}
	In \cite{bhimani_19} the authors proved boundedness results for such spectral functions of the harmonic oscillator (and more general ones of Mikhlin-H\"ormander type) by exploiting the connection with Fourier multipliers on the torus. In particular, they obtained that $m(\cH)$ is bounded on $M^p(\rd)$, $1\le p < \infty$, provided that $\d/\g > d |1/p-1/2|$. 
	
	Such a boundedness result for oscillating multipliers for $\cH$ can be actually improved via pseudo-differential analysis. The main ingredient comes from the analysis of semilinear parabolic equations performed in \cite[Theorem 1.2]{nicola_14}, from which we isolate a special case of interest.
	
	\begin{lemma} \label{lem-nicola} Let $b$ be a real-valued symbol in $S^1$. Given $T>0$, the Schr\"odinger propagator $e^{itb^\w}$ is a pseudo-differential operator whose Weyl symbol belongs to a bounded set of $S^0$, uniformly with respect to $t \in [0,T]$.  
	\end{lemma}
	
	\begin{theorem}\label{thm-hermite-oscmult}
		Given $\gamma\le 1$ and $\delta\ge 0$, for $x>0$ and $t> 0$ consider the spectral multiplier
		\begin{equation}
			m_t(\cH) \coloneqq \cH^{-\delta/2}e^{it\cH^{\gamma/2}}.
		\end{equation} 
		The operator $m_t(\cH)$ is bounded $M^{p,q}_s(\rd)\to M^{p,q}_{s+\delta}(\rd)$ for all $1 \le p,q \le \infty$ and $s \in \bR$. Moreover, if $0< t \le T$, there exists $C(T)>0$ such that
		\begin{equation}
			\| m_t(\cH) f \|_{M^{p,q}_{s+\d}} \le C(T) \| f \|_{M^{p,q}_s}. 
		\end{equation}
	\end{theorem}
	\begin{proof}
		In light of Proposition \ref{prop-powerH}, $\cH^{-\d/2}$ is bounded $M^{p,q}_s(\rd) \to M^{p,q}_{s+\d}(\rd)$. It is then enough to establish the boundedness on $M^{p,q}_s(\rd)$ of the fractional Schr\"odinger flow $e^{it\cH^{\gamma/2}}$. To this aim, Proposition \ref{prop-powerH} shows that $\cH^{\gamma/2}$ is a Weyl operator with symbol in $\Gamma^\gamma \subseteq \Gamma^1 \subset S^1$, so that the assumptions of Lemma \ref{lem-nicola} are satisfied. Boundedness on every modulation space $M^{p,q}_s$ then follows from Proposition \ref{prop-weyl-sm}, along with information about the structure of the operator norm. 
	\end{proof} 
	
	\begin{remark}
		Continuity results for oscillating Hermite  multipliers on Lebesgue and Hardy spaces were proved for instance in \cite{thangavelu_87}, and in \cite{chen} for general Schr\"odinger flows associated with operators whose kernel satisfies a pointwise upper bound of Gaussian type. More recently, $L^p$ bounds for oscillating Hermite multipliers were obtained in \cite{bui_23} under the constraint $\d/\g \ge d|1/p-1/2|$. 
	\end{remark}
	
	Boundedness on modulation spaces of oscillating multipliers for the twisted Laplacian is thus obtained via transference (Theorem \ref{thm-transf}).    
	
	\begin{corollary}\label{cor-oscmult-L}
		Given $\gamma\le 1$ and $\delta\ge 0$, for $x>0$ and $t> 0$ consider the spectral multiplier 
		\begin{equation}
			m_t(\cL) \coloneqq \cL^{-\delta/2}e^{it\cL^{\gamma/2}}.
		\end{equation} 
		The operator $m_t(\cL)$ is bounded on $M^p(\rdd)$ for all $1 \le p \le \infty$. Moreover, if $0<t \le T$, there exists $C(T)>0$ such that
		\begin{equation}
			\| m_t(\cL) f \|_{M^p_{}} \le C(T) \| f \|_{M^p}. 
		\end{equation}
	\end{corollary}

	\section{The heat semigroup for the twisted Laplacian}\label{sec-heat}
	Let us consider now the heat semigroup associated with the twisted Laplacian, namely $e^{-t\cL}$ with $t>0$. By the spectral calculus this is defined by
	\begin{equation}\label{}
		e^{-t\cL}f = \sum_{k = 0}^\infty e^{-(2k+d)t}Q_k f, \qquad f \in L^2(\rdd). 		
	\end{equation}
	It is not surprising that the heat semigroup inherits a twisted convolution structure from \eqref{eq-Qk-def} --- see also \cite{thangavelu_93,wong_03}. Precisely, by linearity we have
	\begin{equation}\label{eq-def-pt}
		e^{-t\cL}f = f \times p_t, \qquad p_t(z) \coloneqq (8\pi)^{-d}\sum_{k = 0}^\infty e^{-(2k+d)t}\vp_k(z).
	\end{equation}
	The generating function identity \eqref{eq-laguerre-genfun} can be called into play to obtain an explicit form for the twisted heat propagator:
	\begin{align*}
		p_t(z) & = (8\pi)^{-d} e^{-dt} \sum_{k=0}^\infty (e^{-2t})^k L_k^{d-1}\Big( \frac{|z|^2}{2}\Big) e^{-\frac{|z|^2}{4}} \\ 
		& = (8\pi)^{-d} e^{-dt} (1-e^{-2t})^{-d}\exp\Big( - \frac{1+e^{-2t}}{1-e^{-2t}} \frac{|z|^2}{4}\Big) \\
		& = (16\pi \sinh t)^{-d} \exp\Big(-\frac 1 4 \coth(t)|z|^2 \Big).
	\end{align*}
	
	In order to prove bounds for $e^{-t\cL}f=f\times p_t$ on modulation spaces, in light of Remark \ref{rem-modwp} we first perform a time-frequency analysis of the twisted propagator. 
	\begin{lemma}
		For $1\le p,q \le \infty$ and $t>0$ we have
		\begin{equation}\label{eq-pt-wpq}
			\| p_t \|_{W^{p,q}} \le C e^{-td}(1+\coth t)^{d/p} (1+\tanh t)^{d/q},
		\end{equation} for a suitable constant $C>0$ that depends only on $d,p,q$. 
	\end{lemma}
	\begin{proof}
		Set $g(z)= e^{-|z|^2}$, $z \in \rdd$, and for $\lambda >0$ consider $g_{\lambda}(z) \coloneqq g (\sqrt{\lambda} z)=e^{-\lambda |z|^2}$. It is then clear that 
		\[ p_t (z) = (16 \pi \sinh t)^{-d} g_{\frac{\coth t}{4}}(z). \]
		As a result, the problem boils down to obtaining bounds for the amalgam norms of dilated Gaussian functions. Straightforward computations yield
		\begin{equation}
			|\cV_g g_{\lambda} (z,\zeta)| = (1+\lambda)^{-d}e^{ -\frac{\lambda}{1+\lambda}|z|^2} e^{-\frac{4}{1+\lambda}|\zeta|^2}, \qquad (z,\zeta) \in \rdq,
		\end{equation} and taking mixed Lebesgue norms gives
		\begin{equation}
			\| g_{\lambda} \|_{W^{p,q}} = C_{d,p,q} \lambda^{-d/q} (1+\lambda)^{d(1/q+1/p-1)},
		\end{equation} where $C_{d,p,q}=4^{-d/p} \pi^{d(1/p+1/q)} p^{-d/p}q^{-d/q}$. We finally obtain the claim with elementary bounds after setting $\lambda= (\coth t) /4$. 
	\end{proof}
	
	Boundedness of $e^{-t\cL}$ on modulation and spaces holds as detailed below. In the statement, we agree that $X^{p,q}$ denotes either $M^{p,q}$ or $W^{p,q}$. 
	\begin{theorem} \label{maint-heat1}
		Consider $1 \le p_1,p_2,q_1,q_2 \le \infty$. The special Hermite heat semigroup $e^{-t\cL}$, $t>0$, is a continuous map $\X1(\rdd) \to \X2(\rdd)$ if and only if $q_2 \ge q_1$. In such case, the following bound holds:
		\begin{equation}\label{eq-maint1}
			\| e^{-t\cL}f \|_{\X2} \le C(t) \| f \|_{\X1}, 
		\end{equation} where
		\begin{equation}\label{eq-ct}
			C(t) \coloneqq \begin{cases}
				C e^{-td} & (t\ge 1) \\ C t^{-d/\tilde{p}} & (0<t\le 1), 
			\end{cases} \qquad \frac{1}{\tilde{p}} \coloneqq \max\Big\{ \frac{1}{p_2}-\frac{1}{p_1},0 \Big\}, 
		\end{equation} and $C>0$ is a constant that does not depend on $t$ or $f$. 
	\end{theorem}
	
	\begin{proof}
		Given $p_1,q_1,p_2,q_2$ with $q_1\le q_2$, let $p_0,q_0 \in [1,\infty]$ be indices such that
		\[ 	\frac{1}{p_0} = \frac{1}{\tilde{p}} = \max\Big\{ \frac{1}{p_2} - \frac{1}{p_1},0 \Big\}, \qquad	q_2\ge q_0, \qquad \frac{1}{q_0} \ge 1 + \frac{1}{q_2}-\frac{1}{q_1}. \]
		Therefore, in light of Proposition \ref{prop-cwt} and Remark \ref{rem-modwp}, we have
		\begin{equation*}
			\| e^{-t\cL}f \|_{\X2} = \| f \times p_t \|_{\X2} \le C \| f \|_{\X1}  \|p_t\|_{\W0},
		\end{equation*} the condition $q_1 \le q_2$ being also necessary in view of Proposition \ref{prop-sharpiff}. We thus resort to \eqref{eq-pt-wpq} with $p=p_0$ and $q=q_0$, and elementary bounds imply that $\|p_t\|_{\W0} \lesssim e^{-td}$ if $t \ge 1$, while $\|p_t\|_{\W0} \lesssim t^{-d/p_0}$ if $0< t \le 1$.
	\end{proof}
	
	\begin{remark}\label{rem-wfs} In passing, we highlight some aspects of the microlocal analysis of the heat flow $e^{-t\cL}$. To this aim, we recall the notion of \emph{Gabor wave front set} introduced by H\"ormander in \cite{horm-quad-91}, then recently rediscovered and further developed in \cite{rodino_14} --- see also \cite{rodino_21} for a plain introduction with some historical notes. Roughly speaking, the Gabor wave front set $WF(u)\subseteq \bR^{2n}\smo$ detects the directions in phase space along which a temperate distribution $u \in \cS'(\bR^n)$ lacks of Schwartz regularity, as measured by the decay of the Gabor transform over a cone. To be precise, we have that $z_0 \notin WF(u)$ if there exists an open conic subset $\Gamma_{z_0} \subseteq \bR^{2n} \smo$ such that $z_0 \in \Gamma_{z_0}$ and, for some (in fact any) $g \in \cS(\bR^n)\smo$,
		\begin{equation}\label{eq-gabor-wf}
			\sup_{z \in \Gamma_{z_0}} |V_g u (z)| < \infty, \quad \forall N \in \bN. 
		\end{equation}
		The heat flow $e^{-t\cL}$ is a pseudo-differential operator with Weyl symbol (cf.\ \cite[Theorem 4.2]{horm-mehler-95})
		\begin{equation}
			e^{-t\cL} = \Theta_t^\w, \qquad \Theta_t(z,\zeta) \coloneqq (8\pi^2 \cosh t)^{-d}e^{-(\tanh t)|\zeta-Jz/2|^2}, \quad (z,\zeta) \in \rdq. 
		\end{equation}
		The symbol $\Theta_t$ belongs to the H\"ormander class $S^0(\rdq)$. A property that distinguishes the Gabor wave front set from other similar notions is precisely the microlocalization of this ``tough'' symbol class, in the sense of the following inclusion: 
		\begin{equation}
			WF(e^{-t\cL}u) = WF(\Theta_t^\w u) \subseteq WF(u), \qquad u \in \cS'(\rdd), \quad t >0. 
		\end{equation}
		This result can be further refined if one takes into account the notion of \emph{singular space}, introduced in \cite{hitrik,pravda_11} to investigate the hypoelliptic features of non-elliptic operators. The singular space associated with $e^{-tL^\w}$ (cf.\ \eqref{eq-symbol-L} and \eqref{eq-bL}) can be readily determined:  
		\begin{equation}
			S_\cL \coloneqq \ker \bL = \{ (z,Jz/2) : z \in \rdd \}. 
		\end{equation}
		Then, \cite[Theorem 6.2]{pravda_18} shows that the phase space singularities that fall outside the singular space are suppressed by the diffusion flow, as a consequence of the following inclusion of Gabor wave front sets:
		\begin{equation}
			WF(e^{-t\cL}u) \subseteq (WF(u) \cap S_{\cL})\smo, \quad u \in \cS'(\rdd), \, t>0. 
		\end{equation}  
		It is thus clear that the heat flow $e^{-t\cL}$ regularizes every initial datum whose Gabor wave front set is disjoint from the singular space $S_\cL$ --- this happens for instance if $u=\delta$ or $u=1$, since $WF(\delta) = \{0\}\times (\rdd \smo)$ and $WF(1) = (\rdd \smo) \times \{0\}$. It would be interesting to determine a distribution $v$ (if any) such that $WF(v)=\cS_L$ --- see for instance \cite[Theorem 6.1]{schulz} in this connection. 
	\end{remark}
	
	\subsection{The fractional heat semigroup}
	Let us now discuss boundedness results for the fractional heat multiplier $e^{-t\cL^\nu}$ with $0<\nu<1$ and $t>0$. The functional calculus for $\cL$ leads us to consider
	\begin{equation}\label{eq-def-fractheatL}
		e^{-t\cL^\nu} f = \sum_{k=0}^{\infty} e^{-t(d+2k)^\nu} Q_k f = f \times \Big((8\pi)^{-d} \sum_{k=0}^{\infty} e^{-t(d+2k)^\nu} \vp_k \Big), \qquad f \in L^2(\rdd). 
	\end{equation} Except for the case $\nu=1$ already treated, it seems not possible to obtain an explicit, closed form for the twisted kernel $p_t^{(\nu)}$ associated with $e^{-t\cL^\nu}$ --- namely, the function such that $e^{-t\cL^\nu}= f \times p_t^{(\nu)}$. 
	
	An alternative representation for the fractional heat semigroup can be obtained via subordination. Let $\eta_t \geq 0$ be the density function of the distribution of the $\nu$-stable subordinator at time $t$ --- see e.g., \cite{bogdan_09} for further details. Then, by construction, $\eta_t(s)=0$ for $s\leq 0$ and we have the identity
	\begin{equation}\label{eq-subord-prop}
		\irp e^{-us} \, \eta_t(s) \, \dd{s}=e^{-tu^{\nu}}, \qquad \forall \, u \ge 0.
	\end{equation}
	The fractional twisted Laplacian semigroup $e^{-t\cL^{\nu}}$ is thus subordinated to the twisted heat semigroup via the following pointwise identity. 
	\begin{proposition}\label{prop-gamma-fract} For all $0<\nu<1$ and $f \in \cS(\rdd)$, 
		\begin{equation}\label{eq-sub-fractheatL}
			e^{-t\cL^{\nu}}f(z)=\irp (e^{-sL} f(z)) \, \eta_t(s) \dd{s}, \qquad z \in \rdd. 
		\end{equation} 
	\end{proposition}
	\begin{proof}
		Let us expand $f$ with respect to special Hermite functions, so that $f = \sum_{\a,\b} c_{\a,\b} \Phi_{\a,\b}$, where $c_{\a,\b} = \lan f, \Phi_{\a,\b} \ran$ and $\a,\b \in \bN^d$. Such an expansion converges uniformly and absolutely to $f$ in $\rdd$, as a consequence of the fact that $\|\Phi_{\a,\b}\|_{L^\infty} = |\Phi_{\a,\b}(0)|\le 1$ and, for all $n \in \bN$,
		\begin{equation}
			|c_{\a,\b}| = |\lan f,\Phi_{\a,\b}\ran| = \frac{|\lan \cH^n f,\Phi_{\a,\b} \ran|}{(d+|\a|+|\b|)^n} \le \frac{\|\cH^n f\|_{L^2}}{(d+|\a|+|\b|)^n},
		\end{equation} where we used that $\cH$ is a symmetric operator --- it is then enough to choose $n$ sufficiently large to ensure convergence. The same argument yields the uniform convergence in $\rdd$ of any $L^\infty$ multiplier expansion 
		\begin{equation}
			m(\cL)f = \sum_{\a,\b \in \bN^d} m(d+2|\b|)  c_{\a,\b} \Phi_{\a,\b},\qquad f \in \cS(\rdd).
		\end{equation}
		On this basis, resorting to the subordination identity \eqref{eq-subord-prop} we have, for all $z \in \rdd$, 
		\begin{align*}
			\irp (e^{-s\cL} f(z)) \, \eta_t(s) \dd{s} & = \irp \Big( \sum_{\a,\b \in \bN^d} e^{-s(d+2|\b|)^\nu} c_{\a,\b} \Phi_{\a,\b}(z) \Big) \eta_t(s) \dd{s} \\
			& = \sum_{\a,\b \in \bN^d} c_{\a,\b} \Phi_{\a,\b}(z) \Big( \irp e^{-s(d+2|\b|)^\nu} \eta_t(s) \dd{s} \Big) \\
			& = \sum_{\a,\b \in \bN^d} e^{-t(d+2|\b|)^\nu} c_{\a,\b} \Phi_{\a,\b}(z) \\ 
			& = e^{-t\cL^\nu} f(z). \qedhere
		\end{align*}
	\end{proof}
	
	In the following statement, we agree that $X^{p,q}$ denotes either $M^{p,q}$ or $W^{p,q}$.
	\begin{theorem}\label{maint-heatfrac} Consider $0<\nu <1$ and $1 \le p_1,p_2,q_1,q_2 \le \infty$. The operator $e^{-t\cL^\nu}$ defined in \eqref{eq-def-fractheatL} is a continuous map $\X1(\rdd) \to \X2(\rdd)$ if and only if $q_2 \ge q_1$ and, in such case, the following bound holds:
		\begin{equation}\label{eq-maint-fract}
			\| e^{-t\cL^\nu}f \|_{\X2} \le C_\nu(t) \| f \|_{\X1}, 
		\end{equation} where
		\begin{equation}\label{eq-ct-fract}
			C_\nu(t) \coloneqq \begin{cases}
				C e^{-td^\nu} & (t\ge 1) \\ C t^{-\frac{d}{\nu \tilde{p}}} & (0<t\le 1), 
			\end{cases} \qquad \frac{1}{\tilde{p}} \coloneqq \max\Big\{ \Big( \frac{1}{p_2}-\frac{1}{p_1}\Big),0 \Big\}, 
		\end{equation} and $C>0$ is a constant that does not depend on $t$ or $f$.
		
	\end{theorem}
	\begin{proof}
		First, let us discuss the necessity of the condition $q_2 \ge q_1$. Recall that $e^{-t\cL^\nu} f = f \times p_t^{(\nu)}$. Even if an explicit formula for $p_t^{(\nu)}$ is not currently available, a rough bound shows that $p_t^{(\nu)} \in M^1(\rdd)$. Indeed, by virtue of Proposition \ref{prop-polybound},
		\begin{equation}
			\| p_t^{(\nu)} \|_{M^1}  = \| p_t^{(\nu)} \|_{W^1}  \le \sum_{k=0}^{\infty} e^{-t(d+2k)^\nu} \| \vp_k \|_{W_1}.
		\end{equation} Let $N$ be as in Proposition \ref{prop-polybound}. Resorting to \eqref{eq-vpk-wpq}, we have
		\begin{equation}
			\|p_t^{(\nu)}\|_{M^1} \lesssim \Big( \sum_{k=0}^{N-1} e^{-t(d+2k)^\nu} \|\vp_k\|_{W^1} + \sum_{k=N}^{+\infty} e^{-t(d+2k)^\nu}  k^{2d^2-1} \Big). 
		\end{equation} Therefore, 
		\begin{align*}
			\sum_{k=N}^{+\infty} e^{-t(d+2k)^\nu}  k^{2d^2-1}  & \le \irp e^{-t(d+2y)^\nu}(d+2y)^{2d^2-1} \dd{y} \\ & = \frac{e^{-td^\nu}}{\nu} \irp e^{-tu}(d^\nu+u)^{-1+2d^2/\nu} \dd{u},
		\end{align*} where the substitution $(d+2y)^\nu= d^\nu+u$ was performed, and the integral is thus finite for all $t>0$. This shows that $p_t^{(\nu)} \in M^1(\rdd)$ and the necessity of the condition $q_2 \ge q_1$ thus follows by Proposition \ref{prop-sharpiff}.
		
		Let us now separately discuss two cases. 
		
		\noindent \textbf{Case $t \ge 1$.}  In view of the assumptions and \eqref{eq-def-fractheatL}, using Proposition \ref{prop-cwt} with a suitable choice of $p_0$ and $q_0$ (see Remark \ref{rem-modwp}) we have
		\begin{align*}
			\|e^{-t\cL^\nu}\|_{\X2} & \lesssim \sum_{k=0}^{\infty} e^{-t(d+2k)^\nu} \| f \times \vp_k\|_{\X2} \\
			&  \lesssim \| f \|_{\X1} \sum_{k=0}^{\infty} e^{-t(d+2k)^\nu} \|\vp_k\|_{\W0}.
		\end{align*} We argue as above, namely for $N$ as in Proposition \ref{prop-polybound}, we have (recall that $M^1(\rdd)=W^1(\rdd)\subseteq W^{p_0,q_0}(\rdd)$)
		\begin{equation}
			\|e^{-t\cL^\nu}\|_{\X2} \lesssim \| f \|_{\X1} \Big( \sum_{k=0}^{N-1} e^{-t(d+2k)^\nu} \|\vp_k\|_{W^{p,q}} + \sum_{k=N}^{+\infty} e^{-t(d+2k)^\nu}  k^{2d^2-1} \Big). 
		\end{equation} The first sum can be bounded by $e^{-td^\nu}$ up to a positive factor that does not depend on $t$. Concerning the second one, arguing as above we obtain
		\begin{equation}
			\sum_{k=N}^{+\infty} e^{-t(d+2k)^\nu}  k^q  \le  \frac{e^{-td^\nu}}{\nu} \irp e^{-tu}(d^\nu+u)^{-1+2d^2/\nu} \dd{u} \le C e^{-td^\nu},
		\end{equation} where we set $ C = \nu^{-1} \irp e^{-u}(d^\nu+u)^{-1+2d^2/\nu} \dd{u}$ --- the integral appearing above is a decreasing function of $t$, and $\nu C$ coincides with its value at $t=1$.  
		
		\noindent \textbf{Case $0<t\le 1$.} The claim is a direct consequence of Theorem \ref{maint-heat1} and Proposition \ref{prop-gamma-fract}. In particular, we have
		\begin{align*}
			\| e^{-t\cL^\nu} f \|_{\X2} & \le \irp \| e^{-s\cL} f \|_{\X2} \eta_t(s) \dd{s} \\
			& \lesssim \| f\|_{\X1} \Big( \int_0^1  s^{-d/\tilde{p}} \eta_t(s) \dd{s} + \int_1^{+\infty} e^{-sd} \eta_t(s) \dd{s} \Big).
		\end{align*} The second integral is clearly bounded by $e^{-td^\nu}$ in view of \eqref{eq-subord-prop}. Concerning the first one, let us first recall the following identity for the Gamma function: 
		\begin{equation}\label{eq-negpow-id}
			u^{-w} = \frac{1}{\Gamma(w)} \irp e^{-yu} y^{w-1} \dd{y}, \qquad u >0, \quad w \in \bR.  
		\end{equation} Then we infer
		\begin{equation}
			\int_0^1  s^{-d/\tilde{p}} \eta_t(s) \dd{s} = \int_0^1  \Big( \frac{1}{\Gamma(d/\tilde{p})} \irp e^{-ys} y^{d/\tilde{p}-1} \dd{y} \Big)  \eta_t(s) \dd{s}.
		\end{equation} After swapping the integration order and using the identity \eqref{eq-subord-prop} we obtain
		\begin{equation}
			\int_0^1  s^{-d/\tilde{p}} \eta_t(s) \dd{s} \lesssim \irp e^{-ty^\nu} y^{d/\tilde{p}-1} \dd{y}.
		\end{equation}
		The substitution $v=ty^\nu$ finally yields
		\begin{equation}
			\irp e^{-ty^\nu} y^{d/\tilde{p}-1} \dd{y} \lesssim \frac{1}{\nu} \Gamma\Big(\frac{d}{\a\tilde{p}}\Big) t^{-\frac{d}{\nu\tilde{p}}},
		\end{equation}
		therefore 
		\begin{equation}
			\| e^{-t\cL^\nu} f \|_{\X2} \lesssim e^{-td^\nu} + t^{-\frac{d}{\nu\tilde{p}}} \le t^{-\frac{d}{\nu\tilde{p}}}, \qquad 0<t\le 1. \qedhere
		\end{equation}
	\end{proof}
	
	Finally, let us examine the boundedness of the fractional heat semigroup on Lebesgue spaces.
	
	\begin{corollary}\label{cor-fracheatL-leb}
		For $0<\nu\le 1$ and $1\le p \le q \le \infty$ we have
		\begin{equation}
			\|e^{-t\cL^\nu} f \|_{L^q} \le C_{p,q}^{(\nu)}(t) \|f\|_{L^p}, 
		\end{equation} where
		\begin{equation}
			C_{p,q}^{(\nu)}(t) \coloneqq \begin{cases}
				C e^{-td^\nu} & (t\ge 1) \\ C t^{-\mu_\nu} & (0<t\le 1), 
			\end{cases} \qquad \mu_\nu \coloneqq \max \Big\{ \frac{d}{\nu} \Big(\frac{1}{\min\{q,q'\}}-\frac{1}{\max\{p,p'\}} \Big) , 0\Big\},
		\end{equation} for a constant $C>0$ that does not depend on $f$ or $t$.
	\end{corollary}
	\begin{proof} The claim can be proved by combining Theorem \ref{maint-heatfrac} (for amalgam spaces, with $p_2=\min\{q,q'\}$, $q_2=q$, $p_1=\max\{p,p'\}$ and $q_1=p$) with the embeddings in \eqref{eq-embed-lp-amalg}:
		\begin{align*}
			\|e^{-t\cL^\nu} f\|_{L^q} & \lesssim \|e^{-t\cL^\nu} f\|_{W^{p_2,q}} \\ 
			& \le C_{p,q}^{(\nu)}(t) \| f \|_{W^{p_1,p}} \\
			& \lesssim C_{p,q}^{(\nu)}(t) \|f \|_{L^p}, 
		\end{align*} where $C_{p,q}^{(\nu)}(t)$ is the constant given in the claim. 
	\end{proof}
	
	\begin{remark} The previous bound recaptures and extends the one proved in \cite[Theorem 7.4]{wong_05} in the case where $\nu=1$ and $1 \le p \le 2 \le q \le \infty$. More generally, note that the singularity at small time always occurs even if $q=p$ unless $p=2$ --- in particular, $\mu_\nu = \tfrac{d}{\nu}|1-\tfrac{2}{p}|$. We also highlight that, in the case $\nu=1$, the singularity can be unveiled as well via Young's convolution inequality after computing the $L^r$ norm of the heat propagator $p_t$, where $r$ is such that $1/p+1/r=1+1/q$ and $q\ge p$. 
	\end{remark}
	
	\subsection{More on negative powers of the twisted Laplacian} 
	As far as negative powers of $\cL$ are concerned, a boundedness result supplemental to Theorem \ref{cor-fracpow-L} can be derived from Theorem \ref{maint-heat1}. To this aim, inspired by \cite{cappiello,thangavelu_18}, we use again subordination and the Gamma functional calculus to give a representation of $\cL^{-\nu}$ in terms of the heat flow --- the proof goes as in that of Proposition \ref{prop-gamma-fract}.
	
	\begin{lemma}
		For all $\nu > 0$ and $f \in \cS(\rdd)$, we have the pointwise identity 
		\begin{equation}\label{eq-def-negpowL}
			\cL^{-\nu} f(z) = \frac{1}{\Gamma(\nu)} \irp e^{-t\cL} f(z)\, t^{\nu-1} \dd{t}, \qquad z \in \rdd. 
		\end{equation}
	\end{lemma}
	
	As a consequence of Theorem \ref{maint-heat1}, it is now easy to show that $\cL^{-\nu}$ is a continuous map $\M1(\rdd) \to \M2(\rdd)$, $1 \le p_1,p_2,q_1,q_2 \le \infty$, if $q_2 \ge q_1$ and $\nu>d/\tilde{p}$, since
	\begin{align*}
		\| \cL^{-\nu} f \|_{\M2} & \lesssim \irp \| e^{-t\cL} f \|_{\M2} t^{\nu-1} \dd{t} \\
		& \lesssim \| f\|_{\M1} \Big( \int_0^1  t^{\nu-1-d/\tilde{p}} \dd{t} + \int_1^{+\infty} e^{-td} t^{\nu-1} \dd{t} \Big). 
	\end{align*} 
	
	Moreover, arguing as in the proof of Corollary \ref{cor-fracheatL-leb}, boundedness extends to Lebesgue spaces: for $1 \le p \le  q \le \infty$, we have
	\begin{equation}
		\| \cL^{-\nu} f\|_{L^q} \lesssim \|f \|_{L^p}, \qquad \nu > d \max \Big\{ \frac{1}{\min\{q,q'\}}-\frac{1}{\max\{p,p'\}}  , 0\Big\}.
	\end{equation} 
	
	\begin{remark} We emphasize that the same arguments used here, namely subordination and integral representations, can be used to handle a number of operators associated with $\cL$ that are of common use in harmonic analysis, such as the Bessel potentials 
		\begin{equation}
			(I+\cL)^{-\nu}f(z) = \frac{1}{\Gamma(\nu)}\irp t^{\nu-1}e^{-t} e^{-t\cL}f(z) \dd{t}, \qquad \nu >0, 
		\end{equation} or, in connection with the Schr\"odinger case, the Riesz-type means
		\begin{equation}
			I_{u,v}f(z) = uv^{-u} \int_0^v (v-t)^{-u-1} e^{-it\cL}f(z) \dd{t}, \qquad u,v>0. 
		\end{equation}
		We also expect that an improvement of the boundedness result already obtained for positive fractional powers $0<\nu<1$ can be derived from the Gamma representation
		\begin{eqnarray}
			\cL^{\nu} f(z) = \frac{1}{\Gamma(-\nu)} \irp (e^{-t\cL} f(z)-f(z))\, \frac{1}{t^{\nu+1}} \dd{t}, \qquad z \in \rdd,
		\end{eqnarray} although refined bounds for $e^{-t\cL}- I$ rather than just $e^{-t\cL}$ are required. 
	\end{remark}
	
	\section*{Acknowledgements} 
	The author is indebted to Fabio Nicola, Luigi Rodino and Patrik Wahlberg for fruitful discussions on the topics of this note. 
	
	The author is member of Gruppo Nazionale per l’Analisi Matematica, la
	Probabilit\`a e le loro Applicazioni (GNAMPA) --- Istituto Nazionale di Alta Matematica (INdAM). The present research is partially supported by the GNAMPA-INdAM project ``Analisi armonica e stocastica in problemi di quantizzazione e integrazione funzionale'', award number (CUP): E55F22000270001. 
	
	The author reports there are no competing interests to declare.

\end{document}